\documentclass[12pt]{amsart} 
\usepackage[mathscr]{eucal} 
\usepackage{amsmath,amsfonts} 
\parskip=\smallskipamount 
\newtheorem{theorem}{Theorem}[section] 
\newtheorem{proposition}[theorem]{Proposition} 
 
\newtheorem{lemma}[theorem]{Lemma} 
\theoremstyle{definition}

\newtheorem{examples}[theorem]{Examples}
 
\newtheorem{remark}[theorem]{Remark} 
\newtheorem{remarks}[theorem]{Remarks}
\makeatletter 
\@addtoreset{equation}{section} 
\makeatother

\newcommand{\CC}{{\mathbb C}} 
\newcommand{\NN}{{\mathbb N}}

\newcommand{\RR}{{\mathbb R}} 
 
\newcommand{\cA}{{\mathcal A}} 
\newcommand{\cB}{{\mathcal B}}

\newcommand{\cE}{{\mathcal E}} 
\newcommand{\cF}{{\mathcal F}} 
\newcommand{\cG}{{\mathcal G}} 
\newcommand{\cH}{{\mathcal H}} 
\newcommand{\cI}{{\mathcal I}}
\newcommand{\cJ}{{\mathcal J}} 
\newcommand{\cK}{{\mathcal K}} 
\newcommand{\cL}{{\mathcal L}} 
 
\newcommand{\cN}{{\mathcal N}} 
 
\newcommand{\cR}{{\mathcal R}} 
\newcommand{\cS}{{\mathcal S}}

\newcommand{\cV}{{\mathcal V}} 
\newcommand{\cW}{{\mathcal W}} 
\newcommand{\cX}{{\mathcal X}}
\newcommand{\cY}{{\mathcal Y}}

\newcommand{\fk}{\mathbf{k}}

\newcommand{\Ra}{\Rightarrow}

\newcommand{\ra}{\rightarrow} 
\newcommand{\ol}{\overline}

\let\phi=\varphi

\newcommand{\lin}{\operatorname{Lin}}

    \newcommand{\Lloc}{\cB_{\mathrm{loc}}}

\newcommand{\nr}[1]{\vspace{0.1ex}\noindent\hspace*{12mm}\llap{\textup{(#1)}}} 
 
\begin{document} 
\title[Operator Models for Hilbert Locally $C^*$-Modules]{Operator Models for Hilbert Locally $C^*$-Modules}\thanks{Work supported by a grant of the Romanian 
National Authority for Scientific Research, CNCS  UEFISCDI, project number
PN-II-ID-PCE-2011-3-0119.}
 
 \date{\today}
  
\author[A. Gheondea]{Aurelian Gheondea} 
\address{Department of Mathematics, Bilkent University, 06800 Bilkent, Ankara, 
Turkey, \emph{and} Institutul de Matematic\u a al Academiei Rom\^ane, C.P.\ 
1-764, 014700 Bucure\c sti, Rom\^ania} 
\email{aurelian@fen.bilkent.edu.tr \textrm{and} A.Gheondea@imar.ro} 

\begin{abstract} We single out the concept of 
concrete Hilbert module over a locally $C^*$-algebra 
by means of locally bounded operators on certain strictly 
inductive limits of Hilbert spaces. Using this concept, we construct an operator
model for
all Hilbert locally $C^*$-modules and, as an application, we obtain a direct construction of the 
exterior tensor product of Hilbert locally $C^*$-modules. These are obtained as consequences 
of a general dilation theorem for positive semidefinite kernels invariant under an action of a 
$*$-semigroup with values locally bounded operators. As a by-product, we obtain two 
Stinespring type theorems for completely positive maps on locally $C^*$-algebras and 
with values locally bounded operators.
\end{abstract} 

\maketitle 

\section*{Introduction}

The origins of Hilbert modules over locally $C^*$-algebras (shortly, Hilbert locally $C^*$-modules) 
are related to investigations on 
noncommutative analogues of classical topological objects (groups, Lie groups, vector bundles,
index of elliptic operators, etc.) as seen in W.B.~Arveson \cite{Arveson0},
A.~Mallios \cite{Mallios}, D.V.~Voiculescu \cite{Voiculescu},
N.C.~Phillips \cite{Phillips}, 
to name a few. An overview of the theory of Hilbert locally $C^*$-modules can be found in the monograph of M.~Joi\c ta \cite{Joita5}.

This article grew out from the question of understanding Hilbert locally 
$C^*$-modules from the point of view of operator theory, more precisely, dilation of operator
valued kernels. For the case of Hilbert $C^*$-modules, such a point of view was
employed by G.J.~Murphy in \cite{Murphy} and we have been influenced to a large extent by the 
ideas in that article. However, locally $C^*$-algebras and Hilbert modules over locally 
$C^*$-algebras have 
quite involved projective limit structures and our task requires rather different tools and methods. 
The main object to be used in this enterprise is that of a locally bounded operator which, roughly 
speaking, is an adjointable and coherent element of a projective limit of Banach spaces of bounded 
operators between strictly inductive limits of Hilbert spaces (locally Hilbert spaces). 

Briefly, in Example~\ref{ex:rhlm} we single out the concept of 
represented (concrete) Hilbert locally $C^*$-module by locally bounded operators, 
then prove in Theorem~\ref{t:rephlcm}
that this concept makes the operator model for
all Hilbert locally $C^*$-modules and, as an application, we obtain in Theorem~\ref{t:etp}
a direct construction of the 
exterior tensor product of Hilbert locally $C^*$-modules. 
These are obtained as consequences of a general dilation theorem for positive semidefinite 
kernels with values locally bounded operators, presented in both linearisation 
(Kolmogorov decomposition) form and reproducing kernel space form. 
We actually prove in Theorem~\ref{t:dilation}, the main result of this 
article, a rather general dilation theorem for positive semidefinite kernels with values locally 
bounded operators and that are invariant under a left action of a $*$-semigroup. Consequently, 
in addition to the application to Hilbert locally $C^*$-modules explained before, we briefly 
discuss two versions of
Stinespring type dilation theorems for completely positive maps on locally $C^*$-algebras and with
values locally bounded operators.

In the following we describe the contents of this article. In the preliminary section 
we start by reviewing projective limits and inductive
limits of locally convex spaces that make the fabric of this article, point out the similarities as well as 
the main differences, concerning completeness and Hausdorff separation, 
between them and discuss the concept of coherence. 
Then we recall the concept of locally Hilbert 
space and reorganise the basic properties of locally bounded operators: these concepts have been 
already introduced and studied under slightly different names by A.~Inoue \cite{Inoue}, 
M.~Joi\c ta \cite{Joita1}, D.~Ga\c spar, P.~Ga\c spar, and N.~Lupa \cite{GasparGasparLupa} 
and D.J.~Karia and Y.M.~Parma \cite{KariaParmar} but, for our 
purposes, especially those related to tensor products, some of the properties require clarification, 
for example in view of the concept of coherence. 
Finally, we briefly review the concept of locally $C^*$-algebra, their operator model and 
spatial tensor product.

The second section is devoted to positive semidefinite kernels with values locally bounded
operators, where the main issue is related to their locally Hilbert space linearisations 
(Kolmogorov decompositions) and their reproducing kernel locally Hilbert spaces. 
For the special case of kernels invariant under the action of some $*$-semigroups we prove 
the  general dilation result in Theorem~\ref{t:dilation} which provides a necessary and sufficient 
boundedness condition for the existence of invariant locally Hilbert space linearisations, 
equivalently,  existence of invariant reproducing kernel locally Hilbert spaces, in terms of an analog
of the boundedness condition of B.~Sz.-Nagy \cite{SzNagy}. The proof of this theorem is essentially
a construction of reproducing kernel space, similar to a certain extent to that 
used in \cite{Gheondea}, see also \cite{Szafraniec} and the rich bibliography cited there. 
As a by-product we also point out two Stinespring type dilation 
theorems for completely positive maps defined on locally $C^*$-algebras, distinguishing the 
coherent case from the noncoherent case, the latter
closely related to \cite{Kasparov} and \cite{Joita2}, but rather different in nature.

In the last section, we first review the necessary terminology around the concept of Hilbert module 
over a locally $C^*$-algebra, then apply Theorem~\ref{t:dilation} to obtain the 
operator model by locally bounded operators and use it to provide a rather direct proof of the 
existence of the exterior tensor product of two Hilbert modules over locally $C^*$-algebras, 
similar to \cite{Murphy}; 
following the traditional construction of the exterior tensor product of Hilbert $C^*$-modules 
as in \cite{Lance}, in \cite{Joita4} this tensor product is formed through a generalisation of 
Kasparov's Stabilisation Theorem \cite{Kasparov}.

Once an operator model becomes available, 
the concept of Hilbert locally $C^*$-module is
much better understood and we think that some of the results obtained in this article will prove their 
usefulness for other investigations in this domain.

\section{Preliminaries}\label{s:p}

In this section we review most of the concepts and results that are needed in this article, starting
with projective and inductive limits of locally convex spaces, cf.\ \cite{Grothendieck}, 
\cite{Kothe}, and \cite{Helemski}, 
then considering the concept of locally Hilbert space and the related concept of 
locally bounded operator, cf.\ \cite{Inoue}, \cite{Joita1}, 
\cite{GasparGasparLupa}. For our purposes, we are especially concerned with 
tensor products of locally Hilbert spaces.
Then we review locally $C^*$-algebras, cf.\ \cite{Inoue}, \cite{Schmudgen}, \cite{Allan}, 
\cite{Apostol}, \cite{Phillips}, and define their spatial tensor product.

\subsection{Projective Limits of Locally Convex Spaces.}\label{ss:pllcs}
A \emph{projective system} of locally 
convex spaces is a pair $(\{\cV_\alpha\}_{\alpha\in A};\{\phi_{\alpha,\beta}\}_{\alpha\leq\beta})$
subject to the following properties:
\begin{itemize}
\item[(ps1)] $(A;\leq)$ is a directed poset (partially ordered set);
\item[(ps2)] $\{\cV_\alpha\}_{\alpha\in A}$ is a net of locally convex spaces;
\item[(ps3)] $\{\phi_{\alpha,\beta}\mid \phi_{\alpha,\beta}\colon 
\cV_\beta\ra \cV_\alpha,\ \alpha,\beta\in A,\ \alpha\leq\beta\}$ is a net of continuous linear 
maps such that $\phi_{\alpha,\alpha}$ is the identity map on $\cV_\alpha$ for all $\alpha\in A$; 
\item[(ps4)] the following \emph{transitivity} condition holds
\begin{equation}\phi_{\alpha,\gamma}
=\phi_{\alpha,\beta}\circ\phi_{\beta,\gamma},\mbox{ for all } \alpha,\beta,\gamma\in A,
\mbox{ such that } \alpha\leq\beta\leq\gamma.
\end{equation}
\end{itemize}

For such a system, its projective limit is defined as follows. 
First consider the vector space 
\begin{equation}\prod_{\alpha\in A}
\cV_\alpha=\{(v_\alpha)_{\alpha\in A}\mid v_\alpha\in\cV_\alpha,\ \alpha\in A\},
\end{equation}
with product topology, that is, the weakest topology which makes the canonical 
projections $\prod_{\alpha\in A}
\cV_\alpha\ra\cV_\beta$ continuous, for all $\beta\in A$. 
Then define $\cV$ as the subspace of $\prod_{\alpha\in A}
\cV_\alpha$ consisting of all nets of vectors $v=(v_\alpha)_{\alpha\in A}$ subject to the following
\emph{transitivity} condition
\begin{equation}\label{e:comp}
\phi_{\alpha,\beta}(v_\beta)=v_\alpha,\mbox{ for all } \alpha,\beta\in A,\mbox{ such that } 
\alpha\leq\beta,
\end{equation}
for which we use the notation
\begin{equation}\label{e:vel}
v=\varprojlim_{\alpha\in A}v_\alpha.
\end{equation}

Further on, for each $\alpha\in A$, define $\phi_\alpha\colon \cV\ra \cV_\alpha$ as the linear map 
obtained by composing the canonical embedding of $\cV$ in $\prod_{\alpha\in A}
\cV_\alpha$ with the canonical projection on $\cV_\alpha$. 
Observe that $\cV$ is a closed subspace 
of $\prod_{\alpha\in A}\cV_\alpha$ and let the topology on $\cV$ be the weakest locally convex 
topology that makes the linear maps $\phi_\alpha\colon \cV\ra \cV_\alpha$ continuous, for all 
$\alpha\in A$.

The pair $(\cV;\{\phi_\alpha\}_{\alpha\in A})$
is called a \emph{projective limit of locally convex spaces} induced by the projective system
$(\{\cV_\alpha\}_{\alpha\in A};\{\phi_{\alpha,\beta}\}_{\alpha\leq\beta})$ and is denoted by
\begin{equation} \cV=\varprojlim_{\alpha\in A}\cV_\alpha.
\end{equation}

With notation as before, a locally convex space $\cW$ and a net of continuous 
linear maps $\psi_\alpha\colon \cW\ra \cV_\alpha$, $\alpha\in A$, are \emph{compatible} with
the projective
system $(\{\cV_\alpha\}_{\alpha\in A};\{\phi_{\alpha,\beta}\}_{\alpha\leq\beta})$ if
\begin{equation} \psi_\alpha=\phi_{\alpha,\beta}\circ \psi_\beta,\mbox{ for all } \alpha,\beta\in A
\mbox{ with } \alpha\leq\beta.
\end{equation}
For such a pair $(\cW;\{\psi_\alpha)\}_{\alpha\in A}$, there always exists a unique continuous linear map $\psi\colon \cW\ra \cV=\varprojlim_{\alpha\in A}\cV_\alpha$ such that 
\begin{equation}\psi_\alpha=\phi_\alpha\circ \psi,\quad \alpha\in A.
\end{equation}

Note that the projective limit $(\cV;\{\phi_\alpha\}_{\alpha\in A})$ defined before is compatible with
the projective system $(\{\cV_\alpha\}_{\alpha\in A};\{\phi_{\alpha,\beta}\}_{\alpha\leq\beta})$ 
and that, in this sense,
the projective limit $(\cV_\alpha;\{\phi_\alpha\}_{\alpha\in A})$ is uniquely
determined by the projective system 
$(\{\cV_\alpha\}_{\alpha\in A};\{\phi_{\alpha,\beta}\}_{\alpha\leq\beta})$.

The projective limit of a projective system of Hausdorff locally convex spaces  is always
Hausdorff and, if all locally convex spaces are complete, then the projective limit 
is complete.

Let $(\cV;\{\phi_\alpha\}_{\alpha\in A})$, $\cV=\varprojlim_{\alpha\in A}\cV_\alpha$, and 
$(\cW;\{\psi_\alpha\}_{\alpha\in A})$, $\cW=\varprojlim_{\alpha\in A}\cW_\alpha$, 
be two projective limits of locally convex spaces indexed by the same poset $A$. 
A linear map $f\colon \cV\ra\cW$ is called \emph{coherent} if
\begin{itemize}
\item[(cpm)] There exists $\{f_\alpha\}_{\alpha\in A}$ a net of linear maps $f_\alpha\colon\cV_\alpha\ra\cW_\alpha$, $\alpha\in A$, such that $\psi_\alpha\circ f=f_\alpha\circ \phi_\alpha$ for all
$\alpha\in A$.
\end{itemize}
In terms of the underlying projective systems 
$(\{\cV_\alpha\}_{\alpha\in A};\{\phi_{\alpha,\beta}\}_{\alpha\leq\beta})$ and 
$(\{\cW_\alpha\}_{\alpha\in A};\{\psi_{\alpha,\beta}\}_{\alpha\leq\beta})$, (cpm) is equivalent with
\begin{itemize}
\item[(cpm)$^\prime$] There exists $\{f_\alpha\}_{\alpha\in A}$ a net of linear maps 
$f_\alpha\colon\cV_\alpha\ra\cW_\alpha$, $\alpha\in A$, such that $\psi_{\alpha,\beta}\circ
f_\beta=f_\alpha\circ \phi_{\alpha,\beta}$, for all $\alpha,\beta\in A$ with $\alpha\leq\beta$.
\end{itemize}
There is a one-to-one correspondence between the class of all coherent linear maps 
$f\colon\cV\ra\cW$ and the class of all nets $\{f_\alpha\}_{\alpha\in A}$ as in (cpm) or, 
equivalently, as in (cpm)$^\prime$.
It is clear that a coherent linear map $f\colon \cV\ra\cW$ is continuous if and only if $f_\alpha$ 
is continuous for all $\alpha\in A$.

\subsection{Inductive Limits of Locally Convex Spaces.}\label{ss:illcs}
An \emph{inductive system} of locally 
convex spaces is a pair $(\{\cX_\alpha\}_{\alpha\in A};\{\chi_{\beta,\alpha}\}_{\alpha\leq\beta})$
subject to the following conditions:
\begin{itemize}
\item[(is1)] $(A;\leq)$ is a directed poset; 
\item[(is2)] $\{\cX_\alpha\}_{\alpha\in A}$ is a net of locally convex spaces;
\item[(is3)] $\{\chi_{\beta,\alpha}\colon\cX_\alpha\ra\cX_\beta\mid \alpha,\beta\in A,\ \alpha\leq\beta\}$ is a net of continuous linear maps such that $\chi_{\alpha,\alpha}$ is the identity map on $\cX_\alpha$ for all $\alpha\in A$;
\item[(is4)] the following transitivity condition holds
\begin{equation}\chi_{\delta,\alpha}=\chi_{\delta,\beta}\circ\chi_{\beta,\alpha},\mbox{ for all } \alpha,
\beta,\gamma\in A\mbox{ with }
\alpha\leq\beta\leq\delta.
\end{equation}
\end{itemize}

Recall that the \emph{locally convex direct sum} $\bigoplus_{\alpha\in A}\cX_\alpha$ is 
the algebraic direct sum, that is, the subspace of the direct product $\prod_{\alpha\in A}$ defined
by all nets $\{x_\alpha\}_{\alpha\in A}$ with finite support, endowed with the strongest locally
convex topology that makes the canonical embedding 
$\cX_\alpha\hookrightarrow\bigoplus_{\alpha\in A}\cX_\beta$ continuous, for all $\beta\in A$.
In the following, we consider $\cX_\alpha$ canonically identified with a subspace of 
$\bigoplus_{\alpha\in A}\cX_\alpha$ and then, let
the linear subspace $\cX_0$ of $\bigoplus_{\alpha\in A}\cX_\alpha$ be defined by
\begin{equation}\label{e:xzero}
\cX_0=\lin\{x_\alpha-\chi_{\beta,\alpha}(x_\alpha)\mid \alpha,\beta\in A,\ \alpha\leq\beta,\ 
x_\alpha\in \cX_\alpha\}.
\end{equation}

The \emph{inductive limit locally convex space} $(\cX;\{\chi_\alpha\}_{\alpha\in A})$ 
of the inductive system of locally convex
spaces $(\{\cX_\alpha\}_{\alpha\in A};\{\chi_{\beta,\alpha}\}_{\alpha\leq\beta})$ is defined as follows.
Firstly,
\begin{equation}\label{e:indlim}
\cX=\varinjlim_{\alpha\in A}\cX_\alpha=
\bigl(\bigoplus_{\alpha\in A}\cX_\alpha\bigr)/\cX_0.
\end{equation}
Then, for arbitrary $\alpha\in A$, 
the canonical linear map $\chi_\alpha\colon\cX_\alpha\ra\varinjlim_{\alpha\in A}\cX_\alpha$ is
defined as the composition of the canonical embedding 
$\cX_\alpha\hookrightarrow\bigoplus_{\beta\in A}\cX_\beta$ with the quotient map
$\bigoplus_{\alpha\in A}\cX_\beta\ra \cX$. The inductive limit topology of
$\cX=\varinjlim_{\alpha\in A}\cX_\alpha$ is the strongest locally convex topology on 
$\cX$ that makes the linear maps $\chi_\alpha$ continuous, for all $\alpha\in A$.

An important distinction with respect to the projective limit is that, under the assumption that all 
locally convex spaces $\cX_\alpha$, $\alpha\in A$, are Hausdorff, 
the inductive limit topology may
not be Hausdorff, unless the subspace $\cX_0$ is closed in $\bigoplus_{\alpha\in A}\cX_\beta$.
Also, in general, the inductive limit of an inductive system of complete locally convex spaces
is not complete. 

With notation as before, a locally convex space $\cY$, together with a net of continuous 
linear maps $\kappa_\alpha\colon \cX_\alpha\ra\cY$, $\alpha\in A$, is \emph{compatible} with
the inductive
system $(\{\cX_\alpha\}_{\alpha\in A};\{\chi_{\beta,\alpha}\}_{\alpha\leq\beta})$ if
\begin{equation} \kappa_\alpha=\kappa_\beta\circ\chi_{\beta,\alpha},\quad 
\alpha,\beta\in A,\ \alpha\leq\beta.
\end{equation}
For such a pair $(\cY;\{\kappa_\alpha)\}_{\alpha\in A}$, 
there always exists a unique continuous linear map 
$\kappa\colon \cY\ra \cX=\varinjlim_{\alpha\in A}\cX_\alpha$ such that 
\begin{equation}\kappa_\alpha=\kappa\circ\chi_\alpha,\quad \alpha\in A.
\end{equation}
Note that the inductive limit $(\cX;\{\chi_\alpha\}_{\alpha\in A})$ is compatible with
$(\{\cX_\alpha\}_{\alpha\in A};\{\chi_{\beta,\alpha}\}_{\alpha\leq\beta})$ 
and that, in this sense, the inductive limit $(\cX;\chi_\alpha\}_{\alpha\in A})$ 
is uniquely determined by the inductive system 
$(\{\cX_\alpha\}_{\alpha\in A};\{\chi_{\beta,\alpha}\}_{\alpha\leq\beta})$.

Let $(\cX;\{\chi_\alpha\}_{\alpha\in A})$, $\cX=\varinjlim_{\alpha\in A}\cX_\alpha$, and 
$(\cY;\{\kappa_\alpha\}_{\alpha\in A})$, $\cY=\varinjlim_{\alpha\in A}\cY_\alpha$, 
be two inductive limits of locally convex spaces. A linear map $g\colon \cX\ra\cY$ is called
\emph{coherent} if
\begin{itemize}
\item[(cim)] There exists $\{g_\alpha\}_{\alpha\in A}$ a net of linear maps 
$g_\alpha\colon\cX_\alpha\ra\cY_\alpha$, $\alpha\in A$, such that 
$g\circ\chi_\alpha= \kappa_\alpha\circ g_\alpha$ for all $\alpha\in A$.
\end{itemize}
In terms of the underlying inductive systems 
$(\{\cX_\alpha\}_{\alpha\in A};\{\chi_{\beta,\alpha}\}_{\alpha\leq\beta})$ and 
$(\{\cY_\alpha\}_{\alpha\in A};\{\kappa_{\beta,\alpha}\}_{\alpha\leq\beta})$, (cim) is equivalent with
\begin{itemize}
\item[(cim)$^\prime$] There exists $\{g_\alpha\}_{\alpha\in A}$ a net of linear maps 
$g_\alpha\colon\cX_\alpha\ra\cY_\alpha$, $\alpha\in A$, such that $\kappa_{\beta,\alpha}\circ
g_\alpha=g_\beta\circ \chi_{\beta,\alpha}$, for all $\alpha,\beta\in A$ with $\alpha\leq\beta$.
\end{itemize}
There is a one-to-one correspondence between the class of all coherent linear maps 
$g\colon\cX\ra\cY$ and the class of all nets $\{g_\alpha\}_{\alpha\in A}$ as in (cim) or, 
equivalently, as in (cim)$^\prime$.
It is clear that a coherent linear map $g\colon \cX\ra\cY$ is continuous if and only 
$g_\alpha\colon\cX_\alpha\ra\cY_\alpha$ 
is continuous for all $\alpha\in A$.

In the following we recall the special case of a \emph{strictly inductive system}. Assume that we have an inductive system $(\{\cX_\alpha\}_{\alpha\in A};
\{\chi_{\beta,\alpha}\}_{\alpha\leq\beta})$ 
of locally convex spaces such that, for all $\alpha,\beta\in A$ with $\alpha\leq\beta$, we
have $\cX_\alpha\subseteq\cX_\beta$, the linear map $\chi_{\beta,\alpha}\colon \cX_\alpha
\hookrightarrow\cX_\beta$ is the inclusion map, $\chi_{\beta,\alpha}(x)=x$ for all $x\in\cX_\alpha$, 
and that the inductive system is \emph{strict} in the sense that the
topology on $\cX_\alpha$ is the same with the induced topology of $\cX_\beta$ on its subspace
$\cX_\alpha$, for all $\alpha,\beta\in A$ with $\alpha\leq\beta$. Then, with notation as in 
\eqref{e:xzero} and \eqref{e:indlim}, observe the canonical identification,
\begin{equation} \varinjlim_{\alpha\in A} \cX_\alpha
=\bigoplus_{\alpha\in A} \cX_\alpha / \cX_0=\bigcup_{\alpha\in A}\cX_\alpha.
\end{equation}
For arbitrary $\alpha\in A$, the canonical map $\chi_\alpha\colon \cX_\alpha\ra\cX$ is the inclusion
map.

Even in the case of a strictly inductive system of Hausdorff locally convex spaces, the inductive
limit locally convex space may not be Hausdorff, cf.\ \cite{Komura}.

\subsection{Locally Hilbert Spaces.} \label{ss:lhs} By definition, $\{\cH_\lambda\}_{\lambda\in
\Lambda}$ is a \emph{strictly inductive system of Hilbert spaces}  if
\begin{itemize}
\item[(lhs1)] $(\Lambda;\leq)$ is a directed poset; 
\item[(lhs2)] $\{\cH_\lambda\}_{\lambda\in\Lambda}$ is a net of Hilbert spaces
$(\cH_\lambda; \langle\cdot,\cdot\rangle_{\cH_\lambda})$, $\lambda\in\Lambda$;
\item[(lhs3)] for each $\lambda,\mu\in\Lambda$ with $\lambda\leq \mu$  we have 
$\cH_\lambda\subseteq\cH_\mu$; 
\item[(lhs4)] for each $\lambda,\mu\in\Lambda$ with $\lambda\leq\mu$ the inclusion map 
$J_{\mu,\lambda}\colon \cH_\lambda\ra\cH_\mu$ is isometric, that is, 
\begin{equation}\langle x,y\rangle_{\cH_\lambda}=\langle x,y\rangle_{\cH_\mu},\mbox{ for all }
x,y\in\cH_\lambda.\end{equation}
\end{itemize} 

\begin{lemma}\label{l:lhs} For any strictly inductive system of Hilbert spaces 
$\{\cH_\lambda\}_{\lambda\in\Lambda}$, its inductive limit 
$\cH=\varinjlim_{\lambda\in\Lambda}\cH_\lambda$ is a Hausdorff locally convex space.
\end{lemma}

\begin{proof} As in Subsection~\ref{ss:illcs},
for each $\lambda\in\Lambda$, letting $J_\lambda\colon\cH_\lambda\ra\cH$ be the inclusion
of $\cH_\lambda$ in $\bigcup\limits_{\lambda\in\Lambda}\cH_\lambda$, the inductive limit
topology on $\cH$ is the strongest locally convex 
topology on $\cH$ that makes the linear maps $J_\lambda$ 
continuous for all $\lambda\in\Lambda$. 

On $\cH$ a canonical inner product $\langle\cdot,\cdot\rangle_\cH$ can be defined as follows: 
\begin{equation}\label{e:lip}
\langle h,k\rangle_\cH=\langle h,k\rangle_{\cH_\lambda},\quad h,k\in\cH,
\end{equation} where $\lambda\in\Lambda$ is any index for which $h,k\in\cH_\lambda$. It
follows that this definition of the inner product is correct and, for each $\lambda\in\Lambda$, 
the inclusion map $J_\lambda\colon(\cH_\lambda;\langle\cdot,\cdot\rangle_{\cH_\lambda})\ra
(\cH;\langle\cdot,\cdot\rangle_{\cH})$ is isometric. This implies that, letting $\|\cdot\|_\cH$ 
denote the norm induced by the inner product $\langle\cdot,\cdot\rangle_\cH$ on $\cH$, the
norm topology on $\cH$ is weaker than the inductive limit topology of $\cH$. Since the norm
topology is Hausdorff, it follows that the inductive limit topology on $\cH$ is Hausdorff as well.
\end{proof}

A \emph{locally Hilbert space}, see \cite{Inoue}, \cite{Joita1}, 
\cite{GasparGasparLupa}, is, by definition,
the inductive limit
\begin{equation}\label{e:injlim} 
\cH=\varinjlim\limits_{\lambda\in\Lambda}\cH_\lambda=\bigcup_{\lambda\in\Lambda}
\cH_\lambda,
\end{equation} of a strictly inductive system $\{\cH_\lambda\}_{\lambda\in\Lambda}$ of 
Hilbert spaces. We stress the fact that, a locally Hilbert space is rather a special type 
of locally convex space and, in general, not a Hilbert space. It is clear that a locally Hilbert space is 
uniquely determined by the strictly inductive system of Hilbert spaces.

\subsection{Locally Bounded Operators.}\label{ss:lbo}
With notation as in Subsection~\ref{ss:lhs}, let $\cH=\varinjlim_{\lambda\in A}\cH_\lambda$ and 
$\cK=\varinjlim_{\lambda\in A}\cK_\lambda$ be two locally Hilbert spaces generated by strictly
inductive systems of Hilbert spaces 
$(\{\cH_\lambda\}_{\lambda\in\Lambda};\{J_{\mu,\lambda}^\cH\}_{\lambda\leq\mu})$ 
and, respectively, 
$(\{\cK_\lambda\}_{\lambda\in\Lambda};\{J_{\mu,\lambda}^\cK\}_{\lambda\leq\mu})$,
indexed on the same directed poset $\Lambda$. A linear map 
$T\colon \cH\ra\cK$ is called a \emph{locally bounded operator} if $T$ is a continuous 
coherent linear map (as defined in Subsection~\ref{ss:illcs}) and adjointable, 
more precisely,
\begin{itemize}
\item[(lbo1)] There exists a net of operators $\{T_\lambda\}_{\lambda\in\Lambda}$, 
with $T_\lambda\in\cB(\cH_\lambda,\cK_\lambda)$ such that 
$TJ_\lambda^\cH=J_\lambda^\cK T_\lambda$ for all $\lambda\in\Lambda$. 
\item[(lbo2)] The net of operators $\{T_\lambda^*\}_{\lambda\in\Lambda}$ 
is coherent as well, that is, $T_\mu^* J_{\mu,\lambda}^\cK=J_{\mu,\lambda}^\cH T_\lambda^*$,
for all $\lambda,\mu\in \Lambda$ such that $\lambda\leq\mu$.
\end{itemize}
We denote by $\Lloc(\cH,\cK)$ the collection of all locally bounded operators $T\colon\cH\ra\cK$.
It is easy to see that $\Lloc(\cH,\cK)$ is a vector space.

\begin{remarks}\label{r:lbo}
(1) The correspondence between $T\in\Lloc(\cH,\cK)$ and the net of operators 
$\{T_\lambda\}_{\lambda\in\Lambda}$ as in (lbo1) and (lb02) is one-to-one. 
Given $T\in\Lloc(\cH,\cK)$,
for arbitrary $\lambda\in\Lambda$ we have $T_\lambda h=Th$, for all $h\in\cH_\lambda$, with the
observation that $Th\in\cK_\lambda$. Conversely, if $\{T_\lambda\}_{\lambda\in\Lambda}$ is a 
net of operators $T_\lambda\in\cB(\cH_\lambda,\cK_\lambda)$ satisfying (lbo2), then letting
$Th=T_\lambda h$ for arbitrary $h\in\cH$, where $\lambda\in\Lambda$ is such that 
$h\in\cH_\lambda$, it follows that $T$ is a locally bounded operator: 
this definition is correct by (lb01).
With an abuse of notation, but which is explained below
and makes perfectly sense, we will use the notation 
\begin{equation}\label{e:tpl} 
T=\varprojlim\limits_{\lambda\in\Lambda}T_\lambda.\end{equation}

(2) Let $T\colon\cH\ra\cK$ be a linear operator. Then $T$ is locally bounded if and only if:
\begin{itemize}
\item[(i)] For all $\lambda \in\Lambda$ we have $T\cH_\lambda\subseteq\cK_\lambda$ and,
letting $T_\lambda:=T|\cH_\lambda\colon\cH_\lambda\ra\cK_\lambda$, $T_\lambda$ is bounded.
\item[(ii)] For all $\lambda,\mu\in\Lambda$ with $\lambda\leq\mu$, we have 
$T_\mu\cH_\lambda\subseteq \cK_\lambda$ and $T_\mu^*\cK_\lambda\subseteq\cH_\lambda$. 
\end{itemize}

(3) The notion of locally bounded operator $T\colon \cH\ra \cK$
coincides with the concept introduced in Section 5 of \cite{Inoue}, with that from 
Definition~1.5 in \cite{Joita1}, with the concept of
"locally operator" as in \cite{GasparGasparLupa}, and with the concept of "operator" at p.~61 in \cite{KariaParmar}, that is,
\begin{itemize}
\item[(a)] there exists a net of operators
$\{T_\lambda\}_{\lambda\in\Lambda}$, with $T_\lambda\in\cB(\cH_\lambda,\cK_\lambda)$ 
for all $\lambda\in\Lambda$;
\item[(b)] $T_\mu J_{\mu,\lambda}^\cH=J_{\mu,\lambda}^\cK T_\lambda$, for all 
$\lambda\leq \mu$;
\item[(c)] $T_\mu P_{\lambda,\mu}^\cH=P_{\lambda,\mu}^\cK T_\mu$, for all 
$\lambda\leq\mu$, where $P^\cH_{\lambda,\mu}$ is the orthogonal projection of $\cH_\mu$ onto
its subspace $\cH_\lambda$.
\item[(d)] for arbitrary $h\in\cH$ we have 
$Th=T_\lambda h$, where $\lambda\in\Lambda$ is any index such that $h\in\cH_\lambda$. 
\end{itemize}

Observe that the relation in (d) is correct: if $h\in\cH_\lambda$ and $h\in\cH_\mu$,
then for any $\nu\in\Lambda$ with $\nu\geq\lambda,\mu$ (since $\Lambda$ is directed, such a 
$\nu$ always exists), by (b) we have
\begin{equation*} J_{\nu,\lambda}^\cK T_\lambda h=T_{\nu}J_{\nu,\lambda}^\cH h
=T_{\nu}J_{\nu,\mu}^\cH h=J_{\nu,\mu}^\cK T_\mu h.
\end{equation*} 

(4) Any locally bounded operator $T\colon\cH\ra\cK$ 
is continuous with respect to the inductive limit topologies of $\cH$ and $\cK$ but, in general, 
it may not be continuous with respect to the norm topologies of $\cH$ and $\cK$. An arbitrary
linear operator $T\in\Lloc(\cH,\cK)$ is continuous with respect to the norm topologies of $\cH$ 
and $\cK$ if and only if, with respect to the notation as in (lbo1) and (lbo2), 
$\sup_{\lambda\in\Lambda}\|T_\lambda\|_{\cB(\cH_\lambda,\cK_\lambda)}<\infty$. In this case,
the operator $T$ uniquely extends to an operator $\widetilde T\in\cB(\widetilde\cH,\widetilde\cK)$, 
where
$\widetilde\cH$ and $\widetilde\cK$ are the Hilbert space completions of $\cH$ and, respectively,
$\cK$, and $\|\widetilde T\|
=\sup_{\lambda\in\Lambda}\|T_\lambda\|_{\cB(\cH_\lambda,\cK_\lambda)}$.
\end{remarks}

For each $\lambda,\mu\in\Lambda$ with $\lambda\leq\mu$, consider the linear map 
$\pi_{\lambda,\mu}\colon \cB(\cH_\mu,\cK_\mu)\ra\cB(\cH_\lambda,\cK_\lambda)$ defined by
\begin{equation}\label{e:plm} 
 \pi_{\lambda,\mu}(T)={J^{\cK\,*}_{\mu,\lambda}}T J^\cH_{\mu,\lambda},\quad 
T\in\cB(\cH_\mu,\cK_\mu).
\end{equation}
Then $(\{\cB(\cH_\lambda,\cK_\lambda)\}_{\lambda\in\Lambda};
\{\pi_{\lambda,\mu}\}_{\lambda\leq\mu})$ is a projective system of Banach spaces and, letting 
$\varprojlim_{\lambda\in\Lambda}\cB(\cH_\lambda,\cK_\lambda)$ denote its locally convex 
projective limit, there is a canonical embeddingg
\begin{equation}\label{e:lbl} 
\Lloc(\cH,\cK)\subseteq\varprojlim_{\lambda\in\Lambda}\cB(\cH_\lambda,\cK_\lambda).
\end{equation}
With respect to the embedding in \eqref{e:lbl}, 
for an arbitrary element $\{T_\lambda\}_{\lambda\in\Lambda}\in
\varprojlim_{\lambda\in\Lambda} \cB(\cH_\lambda,\cK_\lambda)$, the following assertions are
equivalent:
\begin{itemize}
\item[(i)]  $\{T_\lambda\}_{\lambda\in\Lambda}\in\Lloc(\cH,\cK)$. 
\item[(ii)] The axiom (lbo2) holds.
\item[(iii)] For all $\lambda,\mu\in\Lambda$ with $\lambda\leq\mu$, we have 
$T_\mu\cH_\lambda\subseteq \cK_\lambda$ and $T_\mu^*\cK_\lambda\subseteq\cH_\lambda$.
\end{itemize}
As a consequence of \eqref{e:lbl}, $\Lloc(\cH,\cK)$ has a natural locally convex topology, 
induced by the projective limit locally convex topology of 
$\varprojlim_{\lambda\in\Lambda}\cB(\cH_\lambda,\cK_\lambda)$, more precisely, 
generated by the seminorms $\{q_\lambda\}_{\lambda\in\Lambda}$ defined by
\begin{equation}\label{e:qmt}
q_\mu(T)=\|T_\mu\|_{\cB(\cH_\mu,\cK_\mu)},\quad T
=\{T_\lambda\}_{\lambda\in\Lambda}\in\varprojlim_{\lambda\in\Lambda}
\cB(\cH_\lambda,\cK_\lambda).
\end{equation}
Also, it is easy to see that, with respect to the embedding \eqref{e:lbl}, $\Lloc(\cH,\cK)$ is closed
in $\varprojlim_{\lambda\in\Lambda}\cB(\cH_\lambda,\cK_\lambda)$, hence complete.
               
The locally convex space $\Lloc(\cH,\cK)$ can be organised as a projective limit of Banach spaces, 
in view of \eqref{e:lbl}, more precisely, letting 
$\pi_\mu\colon \varprojlim_{\lambda\in\Lambda} \cB(\cH_\lambda,\cK_\lambda)\ra 
\cB(\cH_\mu,\cK_\mu)$ 
be the canonical projection, we first determine the range of $\pi_\mu$. 
To this end, let us consider 
$\Lambda_\mu=\{\lambda\in\Lambda\mid \lambda\leq \mu\}$, the branch of $\Lambda$ determined
by $\mu$, and note that, with the induced order $\leq$, $\Lambda_\mu$ is a directed poset,
that $(\{\cH_\lambda\}_{\lambda\in\Lambda_\mu};
\{J^\cH_{\gamma,\lambda}\}_{\lambda\leq\gamma\leq\mu})$ and
$(\{\cH_\lambda\}_{\lambda\in\Lambda_\mu};
\{J^\cK_{\gamma,\lambda}\}_{\lambda\leq\gamma\leq\mu})$
are strictly inductive systems of Hilbert
spaces such that $\cH_\mu=\bigcup_{\lambda\in\Lambda_\mu}\cH_\lambda
=\varinjlim\limits_{\lambda\in\Lambda_\mu}\cH_\lambda$ and
$\cK_\mu=\bigcup_{\lambda\in\Lambda_\mu}\cK_\lambda
=\varinjlim\limits_{\lambda\in\Lambda_\mu}\cK_\lambda$,
 and that 
$\pi_\mu(\Lloc(\cH,\cK))=\Lloc(\cH_\mu,\cK_\mu)$ is a Banach subspace of 
$\cB(\cH_\mu,\cK_\mu)$. Consequently,
\begin{equation}\label{e:pll}
\Lloc(\cH,\cK)=\varprojlim_{\lambda\in\Lambda} \Lloc(\cH_\lambda,\cK_\lambda).
\end{equation}

To any operator $T\in\Lloc(\cH,\cK)$ 
one uniquely associates an operator 
$T^*\in\Lloc(\cK,\cH)$ 
called the \emph{adjoint} of $T$ and defined as follows: if $T=\varprojlim\limits_{\lambda\in
\Lambda}T_\lambda$ is associated to 
$\{T_\lambda\}_{\lambda\in\Lambda}$ then $T^*=\varprojlim\limits_{\lambda\in\Lambda}
T_\lambda^*$ is associated to the net 
$\{T_\lambda^*\}_{\lambda\in\Lambda}$. Most of the usual algebraic properties of adjoint operators 
in Hilbert spaces remain true, in particular, the classes of \emph{locally isometric}, 
\emph{locally coisometric}, and that of \emph{locally unitary} operators make sense and have, 
to a certain extent, expected properties.

\subsection{Tensor Products of Locally Hilbert Spaces.} \label{ss:tplhs}
We first recall that the Hilbert space tensor product $\cS\otimes \cL$ of two Hilbert spaces 
$\cS$ and $\cL$ is obtained as the Hilbert space completion of the algebraic tensor product space 
$\cS\otimes_{\mathrm{alg}}\cL$, with inner product 
$\langle\cdot,\cdot\rangle_{\cS\otimes_{\mathrm{alg}}\cL}$
defined on elementary tensors by
$\langle s\otimes l,t\otimes k\rangle_{\cS\otimes_{\mathrm{alg}}\cL}
=\langle s,t\rangle_\cS\,\langle l,k\rangle_\cL$ and then 
extended by linearity to $\cS\otimes_{\mathrm{alg}}\cL$.

We also recall that, for two Hilbert spaces $\cS$ and $\cL$ and two operators $X\in\cB(\cS)$ and 
$Y\in\cB(\cL)$, the operator $X\otimes Y\in\cB(\cS\otimes\cL)$ is defined first by letting 
$(X\otimes Y)(s\otimes l)=Xs\otimes Yl$ for arbitrary $s\in\cS$ and $l\in\cL$, then extended by 
linearity to $\cS\otimes_{\mathrm{alg}}Y$, and finally extended by continuity, taking into account 
that $\|X\otimes Y\|=\|X\|\,\|Y\|$. In addition, $(X\otimes Y)^*=X^*\otimes Y^*$, and from here
other expected properties follow in a natural way.

\begin{proposition}\label{p:tplhs}
Let $\cH=\varinjlim\limits_{\lambda\in\Lambda}\cH_\lambda$ and 
$\cK=\varinjlim\limits_{\alpha\in A}\cK_\alpha$ 
be two locally Hilbert spaces, where $\Lambda$ 
and $A$ are two directed posets. Then
$\{\cH_\lambda\otimes\cK_\alpha\}_{(\lambda,\alpha)\in\Lambda\times A}$ can be naturally 
organised as a strictly inductive system of Hilbert spaces.
\end{proposition}

\begin{proof} With notation as in Subsection~\ref{ss:lhs},
we consider $\Lambda\times A$ with the partial order 
$(\lambda,\alpha)\leq (\mu,\beta)$ if $\lambda\leq\mu$ and $\alpha\leq \beta$, for arbitrary 
$\lambda,\mu\in\Lambda$ and $\alpha,\beta\in A$, and observe that, with this order, 
$\Lambda\times A$ is directed. For each $(\lambda,\alpha)\in \Lambda\times A$, 
consider the algebraic tensor product space $\cH_\lambda\otimes_{\mathrm{alg}} \cK_\alpha$ 
with inner product $\langle\cdot,\cdot\rangle_{\lambda,\alpha}$ defined on elementary 
tensors by
\begin{equation*}\langle (h\otimes k),(g\otimes l)\rangle_{\lambda,\alpha}
=\langle h,g\rangle_{\cH_\lambda} \langle k,l\rangle_{\cK_\alpha}, \quad h,g\in\cH_\lambda,\ 
k,l\in\cK_\alpha,\ \lambda\in\Lambda,\ \alpha\in A,
\end{equation*}  and then extended by linearity. 
 Observe that
$\{\cH_\lambda\otimes_{\mathrm{alg}}\cK_\alpha\}_{(\lambda,\alpha)\in\Lambda\times A}$ 
is an inductive system of linear spaces and that
\begin{equation*}
\cH\otimes_{\mathrm{alg}}\cK=\bigcup_{(\lambda,\alpha)\in \Lambda\times A} 
\cH_\lambda\otimes_{\mathrm{alg}}\cK_\alpha.
\end{equation*}

On $\cH\otimes_{\mathrm{alg}}\cK$ there exists a canonical inner product: firstly,
for arbitrary $h,k\in\cH\otimes_{\mathrm{alg}}\cK$, let
\begin{equation}\label{e:lhtpip}
\langle h,k\rangle_{\cH\otimes_{\mathrm{alg}}\cK}=\langle h,k\rangle_{\lambda,\alpha},
\end{equation} where $\lambda\in\Lambda$ and $\alpha\in A$ are such that $h,k\in\cH_\lambda
\otimes_{\mathrm{alg}}\cK_\alpha$, and then extend 
$\langle\cdot,\cdot\rangle_{\cH\otimes_{\mathrm{alg}}\cK}$ to the whole space 
$\cH\otimes_{\mathrm{alg}}\cK$ by linearity, to
a genuine inner product. Let $ \widetilde{\cH\otimes_{\mathrm{alg}}\cK}$ 
be the completion of the inner product space $(\cH\otimes_{\mathrm{alg}}\cK;
\langle\cdot,\cdot\rangle_{\cH\otimes_{\mathrm{alg}}}\cK)$ to a Hilbert space and observe that, 
for any 
$\lambda\in\Lambda$ and $\alpha\in A$, the inner product space 
$(\cH_\lambda\otimes_{\mathrm{alg}}\cK_\alpha;\langle\cdot,\cdot\rangle_{\lambda,\alpha})$ 
is isometrically included in the Hilbert space $ \widetilde{\cH\otimes_{\mathrm{alg}}\cK}$, 
hence we can take the Hilbert space tensor
product $\cH_\lambda\otimes\cK_\alpha$ as the closure of 
$(\cH_\lambda\otimes_{\mathrm{alg}}\cK_\alpha;\langle\cdot,\cdot\rangle_{\lambda,\alpha})$
inside of $ \widetilde{\cH\otimes_{\mathrm{alg}}\cK}$. 
In this way, the inductive system of Hilbert spaces
$\{\cH_\lambda\otimes\cK_\alpha\}_{(\lambda,\alpha)\in\Lambda\times A}$ is strict.
\end{proof}

With notation as in Proposition~\ref{p:tplhs}, the strictly inductive system of Hilbert spaces
$\{\cH_\lambda\otimes\cK_\alpha\}_{(\lambda,\alpha)\in\Lambda\times A}$
gives rise to a locally Hilbert space
\begin{equation}
\cH\otimes_{\mathrm{loc}}\cK=\varinjlim_{(\lambda,\alpha)\in\Lambda\times A} \cH_\lambda
\otimes \cK_\alpha=\bigcup_{(\lambda,\alpha)\in\Lambda\times A} \cH_\lambda\otimes \cK_\alpha,
\end{equation}
that we call the \emph{locally Hilbert space tensor product}.
The natural topology on 
$\cH\otimes_{\mathrm{loc}}\cK$ is considered the inductive limit topology. 
$\cH\otimes_{\mathrm{loc}}\cK$ is equipped with the inner product as in \eqref{e:lhtpip}
and is dense in the Hilbert space $\widetilde{\cH\otimes_{\mathrm{loc}}\cK}$ but, in general, 
they are not the same.

Let $T\in\Lloc(\cH)$ and $S\in\Lloc(\cK)$ be two locally bounded operators and define 
$T\otimes_{\mathrm{loc}}S\colon \cH\otimes_{\mathrm{loc}}\cK\ra \cH\otimes_{\mathrm{loc}}\cK$ 
as follows: if $T=\varprojlim_{\lambda\in\Lambda} T_\lambda$ and 
$S=\varprojlim_{\alpha\in A}S_\alpha$ 
then, observe that $\{T_\lambda\otimes S_\alpha\}_{(\lambda,\alpha)\in \Lambda\times A}$ is a 
projective net of bounded operators, in the sense that it satisfies the following properties:
\begin{itemize}
\item[(i)] $T_\mu\otimes S_\beta$ reduces $\cH_\lambda\otimes\cK_\alpha$ (that is,  
$\cH_\lambda\otimes\cK_\alpha$ is invariant under both $T_\mu\otimes S_\beta$ and its adjoint), 
for all $\lambda\leq \mu$ and $\alpha\leq\beta$.
\item[(ii)] $P_{\cH_\lambda\otimes\cK_\alpha}(T_\mu\otimes S_\beta)|(\cH_\lambda\otimes
\cK_\alpha)=T_\lambda\otimes S_\alpha$ , for all $\lambda\leq\mu$ and $\alpha\leq\beta$.
\end{itemize}
Consequently, we can define $T\otimes_{\mathrm{loc}} S\in\Lloc(\cH\otimes_{\mathrm{loc}}\cK)$
by
\begin{equation} \label{e:teloces}
T\otimes_{\mathrm{loc}}S
=\varprojlim_{(\lambda,\alpha)\in\Lambda\times A} T_\lambda\otimes S_\lambda,
\end{equation}
 and observe that
\begin{equation}\label{e:testares} 
 (T\otimes_{\mathrm{loc}}S)^*=T^*\otimes_{\mathrm{loc}}S^*.
\end{equation}
In particular, if $T=T^*$ and $S=S^*$ then $(T\otimes_{\mathrm{loc}}S)^*
=T\otimes_{\mathrm{loc}}S$ and, if both $T$ and $S$ are locally positive operators, then 
$T\otimes_{\mathrm{loc}}S$ is locally positive as well.

\subsection{Locally $C^*$-Algebras.} \label{ss:lcsa}
A $*$-algebra $\cA$ is called a \emph{locally $C^*$-algebra}
if it has a complete Hausdorff locally convex topology that is induced by a family of 
$C^*$-seminorms, that is,
seminorms $p$ with the property $p(a^*a)=p(a)^2$ for all $a\in\cA$, see \cite{Inoue}. 
Any $C^*$-seminorm $p$ has also the properties $p(a^*)=p(a)$ and $p(ab)\leq p(a)p(b)$ 
for all $a,b\in\cA$, cf.\ \cite{Sebestyen}. 
Locally $C^*$-algebras have been called also \emph{$LMC^*$-algebras} 
\cite{Schmudgen}, \emph{$b^*$-algebras} \cite{Allan}, and \emph{pro $C^*$-algebras} 
\cite{Voiculescu}, \cite{Phillips}. 

If $\cA$ is a locally $C^*$-algebra, let $S(\cA)$ denote the collection of all continuous 
$C^*$-seminorms and note that $S(\cA)$ is a directed poset, with respect to the partial order 
$p\leq q$ if $p(a)\leq q(a)$ for all $a\in\cA$. If $p\in S(\cA)$ then 
\begin{equation}\label{e:cisp} 
\cI_p=\{a\in\cA\mid p(a)=0\}
\end{equation} is a
closed two sided $*$-ideal of $\cA$ and $\cA_p=\cA/\cI_p$ becomes a $C^*$-algebra with respect 
to the $C^*$-norm $\|\cdot\|_p$ induced by $p$, see \cite{Apostol}, more precisely,
\begin{equation}\label{e:tlp} 
\|a+\cI_p\|_p=p(a),\quad a\in\cA.
\end{equation} 
Letting $\pi_p\colon \cA\ra\cA_p$ denote the 
canonical projection, for any $p,q\in S(\cA)$ such that $p\leq q$ there exists a canonical 
$*$-epimorphism of $C^*$-algebras $\pi_{p,q}\colon \cA_q\ra \cA_p$ such that 
$\pi_p=\pi_{p,q}\circ \pi_q$, with respect to which 
$\{\cA_p\}_{p\in S(\cA)}$ becomes a projective system of $C^*$-algebras such that 
\begin{equation}\label{e:alim}
\cA=\varprojlim\limits_{p\in S(\cA)} \cA_p,\end{equation}
see \cite{Schmudgen}, \cite{Phillips}. It is important to stress that this projective limit is taken
in the category of locally convex $*$-algebras and hence all the morphisms are continuous
$*$-morphisms of locally convex $*$-algebras, which make significant differences with respect to
projective limits of locally convex spaces, that we briefly recalled in Subsection~\ref{ss:pllcs}.

An \emph{approximate unit} of a locally $C^*$-algebra
$\cA$ is, by definition, an increasing net $(e_j)_{j\in\cJ}$ 
of positive elements in $\cA$ with $p(e_j)\leq 1$ for any 
$p\in S(\cA)$ and any $j\in\cJ$, satisfying 
$p(x-xe_j)\xrightarrow[j]{} 0$ and $p(x-e_jx)\xrightarrow[j]{} 0$ for 
all $p\in S_*(\cA)$ and all $x\in \cA$. For any locally $C^*$-algebra,
there exists an approximate unit, cf.\ \cite{Inoue}.

Letting $b(\cA)=\{a\in\cA\mid \sup_{p\in S(\cA)} p(a)<+\infty\}$, it follows that 
$\|a\|=\sup_{p\in S(\cA)} p(a)$ is a $C^*$-norm on the $*$-algebra $b(\cA)$ and, with respect to
this norm, $b(\cA)$ is a $C^*$-algebra, dense in $\cA$, see \cite{Apostol}. The elements
of $b(\cA)$ are called \emph{bounded}.

\begin{examples}\label{ex:lcalhs} Let
$\cH=\varinjlim\limits_{\lambda\in\Lambda}\cH_\lambda$
be a locally Hilbert space and $\Lloc(\cH)$ be the locally convex space of all locally 
bounded operators $T\colon \cH\ra\cH$, see Subsection~\ref{ss:lbo}. 

(1) In the following we show that $\Lloc(\cH)$ is  a locally $C^*$-algebra. Actually, we specialise 
\eqref{e:plm}--\eqref{e:pll} for $\cH=\cK$ and point out what additional structure we get. 
We first observe that $\Lloc(\cH)$ has a natural product and a natural involution $*$, 
with respect to which it is a $*$-algebra.
For each $\mu\in\Lambda$, consider the $C^*$-algebra 
$\cB(\cH_\mu)$ of all bounded linear operators in $\cH_\mu$ and 
$\pi_\mu\colon\Lloc(\cH)\ra\cB(\cH_\mu)$ be the canonical map: for any 
$T=\varprojlim\limits_{\lambda\in\Lambda}T_\lambda\in\Lloc(\cH)$, we have 
$\pi_\mu(T)=T_\mu$. Similarly as for \eqref{e:pll},
$\pi_\mu(\Lloc(\cH))=\Lloc(\cH_\mu)$ is a $C^*$-subalgebra of $\cB(\cH_\mu)$. 

It follows that $\pi_\mu\colon\Lloc(\cH)\ra\Lloc(\cH_\mu)$ is a $*$-epimorphism of $*$-algebras 
and, for each 
$\lambda,\mu\in \Lambda$ with $\lambda\leq\mu$, 
there is a unique $*$-epimorphism of $C^*$-algebras 
$\pi_{\lambda,\mu}\colon \Lloc(\cH_\mu)\ra\Lloc(\cH_\lambda)$, such that 
$\pi_\lambda=\pi_{\lambda,\mu}\pi_\mu$. More precisely, compare with \eqref{e:plm} 
and the notation as in Subsection~\ref{ss:lbo},
$\pi_{\lambda,\mu}$ is the compression of $\cH_\mu$ to $\cH_\lambda$,
\begin{equation}\label{e:plmh} 
\pi_{\lambda,\mu}(S)=J_{\mu,\lambda}^*SJ_{\mu,\lambda},\quad S\in\Lloc(\cH_\mu).
\end{equation}
Then $(\{\Lloc(\cH_\lambda)\}_{\lambda\in\Lambda};\{\pi_{\lambda,\mu}\}_{\lambda,\mu\in\Lambda,\ \lambda\leq \mu})$ is
a projective system of $C^*$-algebras, that is,
\begin{equation}\label{e:ple} 
\pi_{\lambda,\eta}=\pi_{\lambda,\mu}\circ\pi_{\mu,\eta},\quad \lambda,\mu,\eta\in\Lambda,\
\lambda\leq\mu\leq\eta,
\end{equation} and, in addition,
\begin{equation}\label{e:coh}
\pi_\mu(S)P_{\lambda,\mu}=P_{\lambda,\mu}\pi_\mu(S),\quad \lambda,\mu\in\Lambda,\ 
\lambda\leq\mu,\ S\in\Lloc(\cH_\mu),
\end{equation}
such that
\begin{equation}\label{e:loclhs} \Lloc(\cH)=\varprojlim_{\lambda\in\Lambda} \Lloc(\cH_\lambda),
\end{equation}
hence $\Lloc(\cH)$ is a locally $C^*$-algebra.

For each $\mu\in \Lambda$, letting $p_\mu\colon\Lloc(\cH)\ra\RR$ be defined by 
\begin{equation}\label{e:psmu}
 p_\mu(T)=\|T_\mu\|_{\cB(\cH_\mu)},\quad T=\varprojlim_{\lambda\in\Lambda}
T_\lambda\in\Lloc(\cH),
\end{equation} then $p_\mu$ is a $C^*$-seminorm on $\Lloc(\cH)$. Then $\Lloc(\cH)$ 
becomes a unital locally $C^*$-algebra with the topology induced by  
$\{p_\lambda\}_{\lambda\in\Lambda}$. 

The $C^*$-algebra $b(\Lloc(\cH))$ coincides with the set of all locally bounded operators 
$T=\varprojlim_{\lambda\in\Lambda}T_\lambda$ such that $\{T_\lambda\}_{\lambda\in\Lambda}$
is uniformly bounded, in the sense that $\sup_{\lambda\in\Lambda}\|T_\lambda\|<\infty$, 
equivalently, those locally bounded operators $T\colon \cH\ra\cH$ that are bounded with respect
to the canonical norm $\|\cdot\|_\cH$ on the pre-Hilbert space 
$(\cH;\langle\cdot,\cdot\rangle_\cH)$. In particular $b(\cA)$ is a $C^*$-subalgebra of 
$\cB(\widetilde\cH)$, where $\widetilde\cH$ denotes the completion of 
$(\cH;\langle\cdot,\cdot\rangle_\cH)$ to a Hilbert space.

(2) With notation as in item (1), let $\cA$ be an arbitray closed $*$-subalgebra of $\Lloc(\cH)$.  
On $\cA$ we consider the collection of $C^*$-seminorms $\{p_\mu|\cA\}_{ \mu\in\Lambda
}$, where the seminorms $p_\mu$ are defined as in \eqref{e:psmu}
and note that, with respect to it, $\cA$ is a locally $C^*$-algebra. The embedding 
$\iota\colon\cA\hookrightarrow\Lloc(\cH)$, in addition to being a $*$-monomorphism, has the 
property that, for each $\lambda\in\Lambda$, it induces a faithful $*$-morphism of $C^*$-algebras 
$\iota_\lambda\colon \cA_{p_\lambda}\hookrightarrow \cB(\cH_\lambda)$ such that, 
$\{\iota_\lambda\}_ {\lambda\in\Lambda}$ has the following properties as in Remark~\ref{r:lbo}.(3): 
for any $\lambda\leq \mu$,
\begin{equation}\label{e:ind} 
\iota_\mu(a_\mu)J_{\mu,\lambda}=
J_{\mu,\lambda}\iota_\lambda(a_\lambda),\quad a=\varprojlim_{\eta\in\Lambda} a_\eta\in \cA,
\end{equation}
\begin{equation}\label{e:coh2}
\iota_\mu(a)P_{\lambda,\mu}=P_{\lambda,\mu}\iota_\mu(a),\quad \ a\in\cA_\mu.
\end{equation}

Also, the $C^*$-algebra
$b(\cA)$ of bounded elements of $\cA$ is canonically embedded as a $C^*$-subalgebra of
$\cB(\widetilde{\cH})$, with notation as in the previous example.
\end{examples}

\begin{remark}\label{r:lsa}
With notation as in the previous examples, classes of operators as 
\emph{locally selfadjoint}, \emph{locally 
positive}, \emph{locally normal}, \emph{locally unitary}, \emph{locally orthogonal projection}, etc.\ 
can be defined in a natural fashion and have expected properties. For example, an operator 
$A=\varprojlim_{\lambda\in\Lambda}A_\lambda$ in $\Lloc(\cH)$ is locally selfadjoint if, by definition,
$A_\lambda=A_\lambda^*$ for all $\lambda$, equivalently, 
$\langle Ah,k\rangle_\cH=\langle h,Ak\rangle_\cH$ for all $h,k\in\cH$, equivalently $A=A^*$.
Similarly, an operator 
$A=\varprojlim_{\lambda\in\Lambda}A_\lambda$ in $\Lloc(\cH)$ is locally positive if, by definition,
$A_\lambda\geq 0$ for all $\lambda$, equivalently, 
$\langle Ah,h\rangle_\cH\geq 0$ for all $h\in\cH$. Then, it is easy to see that,
an arbitrary  operator $T\in\Lloc(\cH)$ is locally positive if and only if 
$T=S^*S$ for some $S\in\Lloc(\cH)$.
\end{remark}

Let $\cA=\varprojlim_{\lambda\in\Lambda}\cA_\lambda$ and 
$\cB=\varprojlim_{\lambda\in\Lambda}\cB_\lambda$ be two locally $C^*$-algebras, where 
$(\{\cA_\lambda\}_{\lambda\in\Lambda};\{\pi^\cA_\lambda\}_{\lambda\in\Lambda})$ and 
$(\{\cB_\lambda\}_{\lambda\in\Lambda};\{\pi_\lambda^\cB\}_{\lambda\in\Lambda})$ are the 
underlying $C^*$-algebras and canonical projections, over the same directed poset 
$\Lambda$. A $*$-morphism
$\rho\colon \cA\ra\cB$ is called \emph{coherent} if
\begin{itemize}
\item[(cam)] There exists $\{\rho_\lambda\}_{\lambda\in\Lambda}$ a net of $*$-morphisms
$\rho_\lambda\colon \cA_\lambda\ra\cB_\lambda$, $\lambda\in\Lambda$, such that 
$\pi_\lambda^\cB\circ \rho=\rho_\lambda\circ\pi_\lambda^\cA$, for all $\lambda\in\Lambda$.
\end{itemize}

\begin{remarks} (1)
Observe that any coherent $*$-morphism of locally $C^*$-algebras is continuous: this is a 
consequence of the fact that any $*$-morphism between $C^*$-algebras is automatically
continuous and the projectivity.

(2) With notation as before, a coherent $*$-morphism of locally $C^*$-algebras 
$\rho\colon\cA\ra\cB$ is faithful (one-to-one) if and only if, for all $\lambda\in\Lambda$, the
$*$-morphism $\rho_\lambda\colon\cA_\lambda\ra\cB_\lambda$ is faithful (one-to-one).
\end{remarks}

In case $\cB=\Lloc(\cH)=\varprojlim_{\lambda\in\Lambda}\Lloc(\cH_\lambda)$, 
where $\cH=\varinjlim_{\lambda\in\Lambda}\cH_\lambda$ is a locally Hilbert space, we talk about
a \emph{coherent $*$-representation} $\rho$ of $\cA$ on $\cH$ if $\rho\colon\cA\ra\Lloc(\cH)$
is a coherent $*$-morphism of locally $C^*$-algebras.

A locally $C^*$-algebra $\cA=\varprojlim\limits_{\lambda\in \Lambda} \cA_\lambda$, where 
$\{\cA_\lambda\mid \lambda\in\Lambda\}$ is a projective system of $C^*$-algebras over some 
directed poset $\Lambda$, for which there exists a locally Hilbert space 
$\cH=\varinjlim\limits_{\lambda\in\Lambda}\cH_\lambda$ such that, for each 
$\lambda\in\Lambda$ the $C^*$-algebra $\cA_\lambda$ is a closed $*$-algebra of 
$\cB(\cH_\lambda)$, is called a \emph{represented locally $C^*$-algebra} or a
\emph{concrete locally $C^*$-algebra}. 
Observe that, in this case, the natural embedding of $\cA$ in $\Lloc(\cH)$ is a coherent 
$*$-representation of $\cA$ on $\cH$.

The following analogue of the Gelfand-Naimark Theorem is essentially Theorem~5.1 in \cite{Inoue}.

\begin{theorem} \label{t:gelfandnaimark}
Any locally $C^*$-algebra $\cA$ can be coherently 
identified with some concrete locally $C^*$-algebra, 
more precisely, if $\cA=\varprojlim\limits_{\lambda\in \Lambda} \cA_\lambda$, 
where $\{\cA_\lambda\mid \lambda\in\Lambda\}$ is a projective system of $C^*$-algebras 
over some directed poset $\Lambda$, then there exists a locally Hilbert space 
$\cH=\varinjlim\limits_{\lambda\in\Lambda}\cH_\lambda$ and 
a faithful coherent $*$-representation $\pi\colon\cA\ra\Lloc(\cH)$.
\end{theorem}

We briefly recall the construction in the proof of Theorem~\ref{t:gelfandnaimark}. By the 
Gelfand-Naimark Theorem, for each $\mu\in\Lambda$ there exists a Hilbert space 
$\cG_\lambda$ and a faithful $*$-morphism $\rho_\mu\colon\cA_\mu\ra\cB(\cH_\mu)$. 
For each $\lambda\in\Lambda$ consider the Hilbert space
\begin{equation}\label{e:hlambda}
\cH_\lambda=\bigoplus_{\mu\leq\lambda} \cG_\mu,
\end{equation} and, identifying $\cH_\lambda$ with the subspace $\cH_\lambda\oplus 0$ of 
$\cH_\eta$, for any $\lambda\leq \eta$, observe that $\{\cH_\lambda\mid \lambda\in\Lambda\}$ 
is a strictly inductive system of Hilbert spaces. Then, for each $\lambda\in\Lambda$ define 
$\pi_\lambda\colon\cA_\lambda\ra\cB(\cH_\lambda)$ by
\begin{equation}\label{e:pilambda}
\pi_\lambda(a)=\bigoplus_{\mu\leq \lambda} \rho_\mu(a_\mu),\quad a=\varprojlim_{\eta\in
\Lambda}a_\eta \in\cA,
\end{equation} and observe that $\{\pi_\lambda\mid \lambda\in\Lambda\}$ is a projective
system of faithful $*$-morphisms, in the sense of \eqref{e:ple} and \eqref{e:coh}. Therefore, the projective limit
$\pi=\varprojlim\limits_{\lambda\in\Lambda} \pi_\lambda\colon\cA\ra\Lloc(\cH)$
is correctly defined and a coherent faithful $*$-representation of $\cA$ on $\cH$.

\subsection{The Spatial Tensor Product of Locally $C^*$-Algebras.}\label{ss:stplcsa} 
Recall that, given two
Hilbert spaces $\cX$ and $\cY$ and letting $\cX\otimes\cY$ denote the Hilbert space tensor 
product,  there is a canonical embedding of the $C^*$-algebra tensor product 
$\cB(\cX)\otimes_* \cB(\cY)$, called the spatial tensor product, as a $C^*$-subalgebra
of the $C^*$-algebra $\cB(\cX\otimes\cY)$, e.g.\ see \cite{Murphy90}.

We first start with two locally
Hilbert spaces $\cH=\varinjlim_{\lambda\in\Lambda}\cH_\lambda$ and 
$\cK=\varinjlim_{\alpha\in A}\cK_\alpha$ and the corresponding locally $C^*$-algebras 
$\Lloc(\cH)=\varprojlim_{\lambda\in\Lambda}\Lloc(\cH_\lambda)$ and 
$\Lloc(\cK)=\varprojlim_{\alpha\in A}\Lloc(\cK_\alpha)$ for which the tensor product 
locally $C^*$-algebra 
$\Lloc(\cH)\otimes_{\mathrm{loc}}\Lloc(\cK)$ is defined by canonically embedding it as a locally 
$C^*$-subalgebra into $\Lloc(\cH\otimes_{\mathrm{loc}}\cK)$, 
where the tensor product
locally Hilbert space $\cH\otimes_{\mathrm{loc}}\cK$ is defined as in Subsection~\ref{ss:tplhs}. 
More precisely, \eqref{e:teloces} provides a canonical embedding of the $*$-algebra 
$\Lloc(\cH)\otimes_{\mathrm{alg}}\Lloc(\cK)$ into
the locally $C^*$-algebra $\Lloc(\cH\otimes_{\mathrm{loc}}\cK)$. For 
$T=\varprojlim_{\lambda\in\Lambda}T_\lambda\in\Lloc(\cH)$ and 
$S=\varprojlim_{\alpha\in A}S_\alpha\in\Lloc(\cK)$, letting 
\begin{equation}p_{\lambda,\alpha}(T\otimes_{\mathrm{loc}}S)
=\|T_\lambda\|_\lambda\, \|S_\alpha\|_\alpha,\quad\lambda\in\Lambda,\ \alpha\in A,
\end{equation} provides a net of cross-seminorms 
$\{p_{\lambda,\alpha}\}_ {\lambda\in\Lambda,\ \alpha\in A}$ on 
$\Lloc(\cH)\otimes_{\mathrm{alg}}\Lloc(\cK)$ that coincides with the net of $C^*$-seminorms
on $\Lloc(\cH\otimes_{\mathrm{loc}}\cK)$, see \eqref{e:psmu}. Consequently, the locally 
$C^*$-algebra tensor product $\Lloc(\cH)\otimes_{\mathrm{loc}}\Lloc(\cK)$ is the completion
of $\Lloc(\cH)\otimes_{\mathrm{alg}}\Lloc(\cK)$ with respect to these seminorms and hence,
canonically embedded into $\Lloc(\cH\otimes_{\mathrm{loc}}\cK)$.

Let $\cA=\varprojlim_{\lambda\in\Lambda}\cA_\lambda$ and 
$\cB=\varprojlim_{\alpha\in A}\cA_\alpha$ be two locally $C^*$-algebras. 
By Theorem~\ref{t:gelfandnaimark}, there exist coherent faithful $*$-representations $\pi\colon
\cA\ra \Lloc(\cH)$ and $\rho\colon\cB\ra\Lloc(\cK)$, for two locally Hilbert spaces 
$\cH=\varinjlim_{\lambda\in\Lambda}\cH_\lambda$ and $\cK=\varinjlim_{\alpha\in A}\cK_\alpha$.
Then $\pi\otimes\rho\colon\cA\otimes_{\mathrm{alg}}
\cB\ra\Lloc(\cH)\otimes_{\mathrm{loc}}\Lloc(\cK)$ is a coherent faithful $*$-morphism. 
We consider the 
represented locally $C^*$-algebras $\pi(\cA)$ in $\Lloc(\cH)$ and $\rho(\cB)$ in $\Lloc(\cK)$ and
make the completion $\pi(\cA)\otimes_{\mathrm{loc}}\rho(\cB)$
of $\pi(\cA)\otimes_{\mathrm{alg}}\rho(\cB)$ within the locally $C^*$-algebra
$\Lloc(\cH)\otimes_{\mathrm{loc}}\Lloc(\cK)$ and then define the \emph{spatial locally $C^*$-
algebra 
tensor product} $\cA\otimes_*\cB$ by identifying it, through the coherent $*$-homomorphism
$\pi\otimes\rho$, with $\pi(\cA)\otimes_{\mathrm{loc}}\rho(\cB)$.

\section{Dilations}\label{s:d}

This is the main section of this article. The object of investigation is the concept of kernel with
values locally bounded operators and that is invariant under an action of a $*$-semigroup
and the main result refers to those
positive semidefinite kernels that provide $*$-representations of the $*$-semigroup on their locally 
Hilbert space linearisations, equivalently on reproducing kernel locally Hilbert space. When 
specialising to completely positive maps on locally $C^*$-algebras and with values locally bounded
operators, we point out how two Stinespring dilation type theorems follow from here.

\subsection{Positive Semidefinite Kernels.} \label{ss:psk}
Let $X$ be a nonempty set and 
$\cH=\varinjlim\limits_{\lambda\in\Lambda}\cH_\lambda$ be a locally Hilbert space, for 
some directed poset $\Lambda$. A map $\fk\colon X\times X\ra \Lloc(\cH)$ is called a \emph{locally 
bounded operator valued kernel} on $X$. Equivalently, with notation as in subsections \ref{ss:lbo} 
and \ref{ss:lcsa}, there exists a projective system 
$\{\fk_\lambda\mid \lambda\in\Lambda\}$ of kernels $\fk_\lambda\colon X\times X\ra 
\Lloc(\cH_\lambda)$, $\lambda\in\Lambda$, where
\begin{equation}\label{e:kp1} 
\fk_\lambda(x,y)=\fk(x,y)_\lambda, \quad \lambda\in\Lambda,\ x,y\in X,
\end{equation} more precisely, for each $\lambda\in\Lambda$ we have 
$\fk_\lambda(x,y)\in\cB(\cH_\lambda)$
such that 
\begin{equation}\label{e:kp2} 
\fk_\lambda(x,y)P_{\lambda,\mu}=P_{\lambda,\mu}\fk_\lambda(x,y),\quad x,y\in X,\ \lambda\leq\mu,
\end{equation} where $P_{\lambda,\mu}$ is the orthogonal projection of $\cH_\mu$ onto 
$\cH_\lambda$, 
and, for any $h\in\cH$,
\begin{equation}\label{e:fk} 
\fk(x,y)h=\fk_\lambda(x,y)h,\quad x,y\in X,
\end{equation}
where $\lambda\in\Lambda$ is such that $h\in\cH_\lambda$.

Given $n\in\NN$, the kernel $\fk\colon X\times X\ra\Lloc(\cH)$ is called 
\emph{$n$-positive semidefinite} if, for any 
$x_1,\ldots,x_n\in X$ and any $h_1,\ldots,h_n\in\cH$, we have
\begin{equation}\label{e:npos} 
\sum_{i,j=1}^n \langle \fk(x_i,x_j)h_j,h_i\rangle_\cH\geq 0.
\end{equation}
It is easy to see that $\fk$ is $n$-positive semidefinite if and only if, for each 
$\lambda\in\Lambda$, the kernel $\fk_\lambda$ is $n$-positive semidefinite.

The kernel $\fk\colon X\times X\ra\Lloc(\cH)$ is called \emph{positive semidefinite} if it is 
$n$-positive semidefinite for all $n\in\NN$. Clearly, this is equivalent with the condition that, 
for each $\lambda\in\Lambda$, the kernel $\fk_\lambda$ is positive semidefinite.

Given a locally bounded operator valued kernel $\fk\colon X\times X\ra \Lloc(\cH)$, with
$\cH=\varinjlim\limits_{\lambda\in\Lambda} \cH_\lambda$,
a \emph{locally Hilbert space linearisation}, also called a \emph{locally Hilbert space Kolmogorov 
decomposition}, of $\fk$ is a pair $(\cK;V)$ such that
\begin{itemize}
\item[(l1)] $\cK=\varinjlim\limits_{\lambda\in\Lambda} \cK_\lambda$ is a locally Hilbert 
space over the same directed poset $\Lambda$.
\item[(l2)] $V\colon X\ra\Lloc(\cH,\cK)$ has the property $\fk(x,y)=V(x)^*V(y)$, for all $x,y\in X$.
\end{itemize}
A linearisation $(\cK;V)$ of $\fk$ is called \emph{minimal} if
\begin{itemize}
\item[(l3)] $V(X)\cH$ is a total subset in $\cK$.
\end{itemize}

\begin{remark}\label{r:mkd} 
From any locally Hilbert space linearisation $(\cK;V)$ of $\fk$, we can obtain a 
minimal one. Indeed, consider $\cK_0$, the closure of the linear subspace generated by 
$V(X)\cH$, which is a locally Hilbert subspace of $\cK$. More precisely, for each 
$\lambda\in\Lambda$, consider $\overline{\lin{V(X)\cH_\lambda}}$, the closure of the linear space 
generated by $V(X)\cH$ as a subspace of $\cK_\lambda$ and observing that 
$\{\overline{\lin{V(X)\cH_\lambda}}\}_{\lambda\in\Lambda}$ is a strictly inductive system of Hilbert
spaces, let 
\begin{equation}\label{e:kzero} 
\cK_0=\varinjlim_{\lambda\in\Lambda}\overline{\lin{V(X)\cH_\lambda}}. 
\end{equation}
For each $\lambda\in\Lambda$, let $J_{\lambda,0}\colon \overline{\lin{V(X)\cH_\lambda}}
\hookrightarrow \cK_\lambda$ be the natural embedding, 
an isometric operator between two Hilbert spaces, and observe that
\begin{equation}\label{e:ceo} 
J_0=\varprojlim_{\lambda\in\Lambda}J_{\lambda,0}\in\Lloc(\cK_0,\cK)
\end{equation}
is an isometric coherent embedding of $\cK_0$ in $\cK$. Then,
$P_0=J_0^*\in\Lloc(\cK,\cK_0)$ is a locally orthogonal projection
of $\cK$ onto $\cK_0$
and then, letting $V_0(x)=P_{\cK_0}V(x)$ for all $x\in X$, 
we obtain a minimal locally Hilbert space linearisation $(\cK_0;V_0)$ of $\fk$. Also, all 
minimal locally Hilbert space linearisations associated to a kernel $\fk$ are unique, modulo locally
unitary equivalence.
\end{remark}

With the same notation as before, let $\cF(X;\cH)$ denote the collection of all maps 
$f\colon X\ra\cH$ and note that it has a natural structure of complex vector space. In addition, 
observe that $\{\cF(X;\cH_\lambda)\}_{\lambda\in\Lambda}$ is a strictly inductive system 
of complex vector spaces, in the sense that $\cF(X;\cH_\lambda)\subseteq \cF(X;\cH_\mu)$ for all $\lambda\leq \mu$, and that
\begin{equation}\label{e:fex}
\cF(X;\cH)=\varinjlim_{\lambda\in\Lambda} \cF(X;\cH_\lambda)
=\bigcup_{\lambda\in\Lambda}\cF(X;\cH_\lambda).
\end{equation}
A complex vector space $\cR$ is called a \emph{reproducing kernel locally Hilbert space} of $\fk$ if
\begin{itemize}
\item[(rk1)] $\cR\subseteq\cF(X;\cH)$, with all algebraic operations, is a locally Hilbert space
$\cR=\varinjlim\limits_{\lambda\in\Lambda}\cR_\lambda$, 
with Hilbert spaces $\cR_\lambda\subseteq\cF(X;\cH_\lambda)$ for all $\lambda\in\Lambda$.
\item[(rk2)] Letting $\fk_x(y)=\fk(y,x)$, $x,y\in X$, we have $\fk_xh\in\cR$ for all 
$x\in X$ and $h\in\cH$.
\item[(rk3)] $\langle f,\fk_xh\rangle_\cR=\langle f(x),h\rangle_\cH$ for all $h\in\cH$, $x\in X$, and 
$f\in \cR$.
\end{itemize}
Observe that, any reproducing kernel locally Hilbert space $\cR$ of $\fk$ has the following 
minimality property as well
\begin{itemize}
\item[(rk4)] $\{\fk_xh\mid x\in X,\ h\in\cH\}$ is total in $\cR$.
\end{itemize}
Also, the reproducing kernels are uniquely determined by their reproducing kernel locally Hilbert 
spaces and, conversely, the reproducing kernel locally Hilbert spaces are uniquely determined by
their reproducing kernels.

We are particularly interested in the relation between locally Hilbert space linearisations and
reproducing kernel locally Hilbert spaces.

\begin{proposition}\label{p:linvsrk}
Let $\fk\colon X\times X\ra\Lloc(\cH)$ be a locally positive semidefinite kernel,  for some locally 
Hilbert space $\cH$ and nonempty set $X$.

\emph{(1)} Any reproducing kernel locally Hilbert space $\cR$ of $\fk$ can be viewed as a minimal
locally Hilbert space linearisation $(\cR;V)$, where $V(x)=\fk_x^*$.

\emph{(2)} For any minimal locally Hilbert space linearisation $(\cK;V)$ of $\fk$, letting
\begin{equation} \cR=\{V(\cdot)^*k\mid k\in\cK\},
\end{equation} we obtain a reproducing kernel Hilbert space $\cR$.
\end{proposition}

The proof is rather straightforward and we omit it, e.g.\ see similar results and their proofs in
\cite{Gheondea} and \cite{AyGheondea}.

\subsection{The General Dilation Theorem.} \label{ss:gdt}
With notation as in the previous subsection, let 
$ S$ be a $*$-semigroup acting on $X$ at left, $ S\times X\ni ( s ,x)
\mapsto  s \cdot x\in X$. A kernel $\fk\colon X\times X\ra\Lloc(\cH)$, for some locally Hilbert
space $\cH$, is called \emph{$ S$-invariant} if
\begin{equation}\fk( s \cdot x,y)=\fk(x, s ^*\cdot y),\quad  s \in S,\ x,y\in X.
\end{equation}
Invariant kernels and their many applications have been considered in mathematical models
of quantum physics \cite{EvansLewis} and (quantum) probability theory \cite{ParthasaratySchmidt}.

\begin{theorem}\label{t:dilation}
Let $ S$ be a $*$-semigroup acting at left on the nonempty set $X$ and let 
$\fk\colon X\times X\ra\Lloc(\cH)$ be a kernel, for some locally Hilbert space 
$\cH=\varinjlim\limits_{\lambda\in\Lambda}\cH_\lambda$. The following assertions
are equivalent:
\begin{itemize}
\item[(1)] The kernel $\fk$ is locally positive semidefinite, invariant under the action of $ S$, 
and
\begin{itemize}
\item[(b)] For any $ s \in S$ and any $\lambda\in\Lambda$, there exists 
$c_\lambda( s )\geq 0$ such that, for any $n\in\NN$, any vectors $h_1,\ldots,h_n\in\cH_\lambda$, 
and any elements $x_1,\ldots,x_n\in X$, we have
\begin{equation*} \sum_{j,k=1}^n\langle \fk( s \cdot x_j, s \cdot x_k)h_k,h_j
\rangle_{\cH_\lambda} \leq c_\lambda( s ) \sum_{j,k=1}^n \langle\fk(x_j,x_k)h_k,h_j
\rangle_{\cH_\lambda}. 
\end{equation*}
\end{itemize}
\item[(2)] There exists a triple $(\cK;\pi;V)$ subject to the following properties:
\begin{itemize}
\item[(il1)] $(\cK;V)$ is a locally Hilbert space linearisation of $\fk$.
\item[(il2)] $\pi\colon S\ra\Lloc(\cK)$ is a $*$-representation.
\item[(il3)] $V( s \cdot x)=\pi( s )V(x)$ for all $ s \in S$ and all $x\in X$.
\end{itemize}
\item[(3)] There exists a reproducing kernel locally Hilbert space $\cR$ with reproducing kernel 
$\fk$ and a $*$-representation $\rho\colon S\ra\Lloc(\cR)$ such that $\fk_{ s \cdot x}=
\rho( s )\fk_x$ for all $ s \in S$ and all $x\in X$.
\end{itemize}
In addition, if this is the case, then the triple $(\cK;\pi;V)$ as in item \emph{(2)} can be chosen 
minimal, in the sense that $\pi( S)V(X)\cH$ is total in $\cK$ and, in this case, it is unique up
to a locally unitary equivalence.
\end{theorem}

\begin{proof} (1)$\Ra$(2). We first fix $\lambda\in\Lambda$ and construct a minimal 
$\cH_\lambda$-valued 
Hilbert space linearisation $(\cK_\lambda;\pi_\lambda;V_\lambda)$ of the positive semidefinite 
kernel $\fk_\lambda \colon X\times X\ra\cB(\cH_\lambda)$. 
Let $\cF(X;\cH_\lambda)$ denote the
complex vector space of functions $f\colon X\ra\cH_\lambda$ and let $\cF_0(X;\cH_\lambda)$ 
denote its subspace of all finitely supported functions. Consider the convolution operator 
$K_\lambda \colon \cF_0(X;\cH_\lambda)\ra \cF(X;\cH_\lambda)$
\begin{equation}
(K_\lambda f)(x)=\sum_{y\in X}\fk_\lambda(x,y)f(y),\quad f\in\cF_0(X;\cH_\lambda),\ x\in X,
\end{equation}
and let $\cG_\lambda\subseteq \cF(X;\cH_\lambda)$ denote its range
\begin{equation}
\cG_\lambda=\{g\in \cF(X;\cH_\lambda)\mid g=K_\lambda f\mbox{ for some }f\in
\cF_0(X;\cH_\lambda)\}.
\end{equation}
On $\cG_\lambda$ a pairing $\langle\cdot,\cdot\rangle_\lambda$ can be defined as follows
\begin{equation}\label{e:lefr}
\langle e,f\rangle_\lambda=
\sum_{x,y\in X}\langle \fk_\lambda(y,x)g(x),h(y)\rangle_{\cH_\lambda},\quad e,f\in \cG_\lambda, 
\end{equation} where $g,h\in\cF_0(X,\cH_\lambda)$ are such that $e=K_\lambda g$ and 
$f=K_\lambda h$. The definition \eqref{e:lefr} is correct and the pairing 
$\langle\cdot,\cdot\rangle_\lambda$ is an inner product on $\cG_\lambda$,
the details are similar with those 
in the proofs of Theorem~3.3 and Theorem~4.2 in \cite{Gheondea}.

Letting $\cK_\lambda$ denote the Hilbert space completion of the pre-Hilbert space 
$(\cG_\lambda;\langle\cdot,\cdot\rangle_\lambda)$, we now show that 
$\{\cK_\lambda\}_{\lambda\in\Lambda}$ can be chosen in such a way that it is 
a strictly inductive system of Hilbert spaces. To see this, we first
observe that, for each $\lambda,\mu\in \Lambda$ with $\lambda\leq\mu$, the pre-Hilbert space 
$\cG_\lambda\subseteq \cG_\mu$ and that, the two inner products 
$\langle\cdot,\cdot\rangle_\lambda$ and $\langle\cdot,\cdot\rangle_\mu$ coincide on 
$\cG_\lambda$. Then, let 
\begin{equation*}\cG=\varinjlim_{\lambda\in\Lambda}\cG_\lambda=\bigcup_{\lambda\in\Lambda}
\cG_\lambda
\end{equation*} be the algebraic
inductive limit, on which we can define an inner product 
$\langle\cdot,\cdot\rangle_\cG$ as follows:
\begin{equation*}\langle g,h\rangle_\cG=\langle g,h\rangle_\lambda,
\end{equation*} where $\lambda\in\Lambda$ is any index with the property that 
$g,h\in\cG_\lambda$. It turns out that this definition is correct, due to the fact that 
$\cG_\lambda\subseteq \cG_\mu$  and that the two inner products 
$\langle\cdot,\cdot\rangle_\lambda$ and $\langle\cdot,\cdot\rangle_\mu$ coincide on 
$\cG_\lambda$,
 for any $\lambda\leq\mu$. Let $\widetilde\cG$ be the Hilbert 
space completion of the inner product space $(\cG;\langle\cdot,\cdot_\cG)$. Then, observe that, 
for each $\lambda\in\Lambda$, the inner product space $(\cG_\lambda;\langle\cdot,\cdot\rangle_
\lambda)$ is isometrically included in $\widetilde\cG$, hence we can take $\cK_\lambda$ as
the closure of $\cG_\lambda$ in $\widetilde\cG$. In this way,
$\{\cK_\lambda\}_{\lambda\in\Lambda}$ is a strictly inductive system of Hilbert spaces hence,
we can let
\begin{equation}\cK=\varinjlim_{\lambda\in\Lambda} \cH_\lambda,
\end{equation} the corresponding locally Hilbert space.

For each $x\in X$, define $V_\lambda(x)\colon \cH_\lambda\ra \cK_\lambda$ by
\begin{equation}\label{e:vex} 
(V_\lambda(x)h)(y)=\fk_\lambda(y,x)h,\quad y\in X,\ h\in \cH_\lambda,
\end{equation} note that the linear operator $V_\lambda(x)$ has its range in $\cG_\lambda$, 
and that
\begin{equation*}\langle V_\lambda(x)h,V_\lambda(x)h\rangle_\lambda
=\langle \fk_\lambda(x,x)h,h\rangle_{\cH_\lambda}
\leq \|\fk_\lambda\| \langle h,h\rangle_{\cH_\lambda},\quad h\in\cH_\lambda,
\end{equation*} hence $V_\lambda(x)\in\cB(\cH_\lambda,\cK_\lambda)$.
In addition, $V_\lambda(x)^*$ is the extension to $\cK_\lambda$ of the evaluation operator 
$\cG_\lambda\ni g\mapsto g(x)\in\cH_\lambda$. This shows that
\begin{equation}V_\lambda(x)^*V_\lambda(y)h=(V_\lambda(y)h)(x)=\fk_\lambda(x,y)h,\quad 
x,y\in X,\ h\in\cH_\lambda.
\end{equation}

For each $ s \in S$ let $\pi_\lambda\colon \cF(X;\cH_\lambda)\ra\cF(X;\cH_\lambda)$
be the linear operator defined by 
\begin{equation}\label{e:pls} 
(\pi_\lambda( s )f)(x)=f( s ^* x),\quad f\in\cF(X;\cH_\lambda),\ x\in X,
\end{equation} and observe that $\pi_\lambda$ leaves the subspace $\cG_\lambda$ invariant. 
Denoting by the same symbol the linear operator 
$\pi_\lambda( s )\colon \cG_\lambda\ra\cG_\lambda$, it follows that 
$\pi_\lambda \colon  S \ra \cL(\cG_\lambda)$ is a $*$-representation of the $*$-semigroup 
$ S$ on the vector space $\cG_\lambda$. In addition, taking into account the 
$ S$-invariance of the kernel $\fk$, and hence of $\fk_\lambda$, we have
\begin{align*} (V_\lambda( s \cdot x)h)(y) & =\fk_\lambda(y, s \cdot x)h
=\fk_\lambda( s ^*\cdot y,x)h=
(V_\lambda(x)h)( s ^*\cdot x)\\
& =(\pi_\lambda( s )V_\lambda(x)h)(y),\quad x,y\in X,\ 
h\in\cH_\lambda,\  s \in S.
\end{align*}

We observe that, due to the boundedness condition (b), for each $ s \in S$,
the linear operator $\pi_\lambda$ is bounded with respect to the norm of the pre-Hilbert space 
$\cG_\lambda$ and hence, it can be uniquely extended to an operator 
$\pi_\lambda( s )\in\cB(\cK_\lambda))$ such that the conditions (il2) and (il3) hold. In addition,
observe that the linear span of $\pi_\lambda( S)V(X)\cH_\lambda$ is $\cG_\lambda$, hence
dense in $\cK_\lambda$.

On the other hand, observe that, for any $\lambda,\mu\in\Lambda$ with $\lambda\leq \mu$ we
have
\begin{equation} V_\mu(x)h=J_{\mu,\lambda}V_\lambda(x)h,\quad x\in X,\ h\in \cH_\lambda,
\end{equation} and, similarly,
\begin{equation} J_{\mu,\lambda}^*\pi_\mu( s )J_{\mu,\lambda}=\pi_\lambda,\quad 
 s \in S.
\end{equation}
Consequently, letting $V\colon X\ra \Lloc(\cH,\cK)$ be defined by
\begin{equation} V(x)h=V_\lambda (x)h,\quad  x\in X,\ h\in \cH, 
\end{equation} where $\lambda\in\Lambda$ is any index such that $h\in\cH_\lambda$ and, 
similarly,
\begin{equation} \pi( s )k=\pi_\mu( s )k,\quad  s \in S,\ k\in\cK,
\end{equation} where $\mu\in\Lambda$ is any index such that $k\in\cK_\mu$, we obtain a triple 
$(\cK;\pi;V)$ with all the required properties.

(2)$\Ra$(3). This is a consequence of Proposition~\ref{p:linvsrk}.

(3)$\Ra$(1). This implication is clear, in view of Proposition~\ref{p:linvsrk}.\end{proof}

The proof of the implication (1)$\Ra$(2) in Theorem~\ref{t:dilation} follows a
reproducing kernel approach. As a 
technical observation, when combining with Proposition~\ref{p:linvsrk}, it shows
that the completion performed at the end of the proof of the implication (1)$\Ra$(2) can be done 
inside of $\cF(X;\cH)$, see also \cite{Szafraniec} for historical comments on this issue.

The boundedness condition (b) is the analogue of the Sz.-Nagy boundedness condition 
\cite{SzNagy} and it is automatic if $S$ is a group with $s^*=s^{-1}$, for all $s\in S$,
see \cite{Szafraniec} for a historical perspective on this issue.
Letting $S=\{e\}$, the trivial group, Theorem~\ref{t:dilation} implies that any positive semidefinite
kernel with values in $\Lloc(\cH)$, for some locally Hilbert space, 
has a locally Hilbert space linearisation, equivalently, is the reproducing kernel of some locally 
Hilbert space of functions defined on $X$ and valued in $\Lloc(\cH)$, a fact observed in 
\cite{GasparGasparLupa}.

\subsection{Completely Positive Maps.}\label{ss:cpm} 
Let $\cA$ be a locally $C^*$-algebra and consider $M_n(\cA)$ the $*$-algebra of $n\times n$ 
matrices with entries in $\cA$. In order to organise it as a locally $C^*$-algebra, we take advantage
of the spatial tensor product defined in Subsection~\ref{ss:stplcsa}, more precisely, we canonically
identify $M_n(\cA)$ with the spatial tensor product locally $C^*$-algebra $M_n\otimes_*\cA$.

Consider now two locally $C^*$-algebras $\cA$ and $\cB$ and let $\phi\colon\cA\ra\cB$ 
be a linear map. For arbitrary $n\in\NN$, consider $\phi_n\colon M_n(\cA)\ra M_n(\cB)$, defined by
\begin{equation}
\phi_n([a_{i,j}]_{i,j=1}^n)=[\phi(a_{i,j})]_{i,j=1}^n,\quad [a_{i,j}]_{i,j=1}^n\in M_n(\cA),
\end{equation}
equivalently, $\phi_n=I_n\otimes \phi$, where $I_n$ is the unit matrix in $M_n$.
Since $M_n(\cA)=M_n\otimes_*\cA$ are locally $C^*$-algebras, it follows that positive 
elements in $M_n(\cA)$ are perfectly defined, hence the cone of positive elements $M_n(\cA)^+$
is defined.
The linear map $\phi$ is called \emph{$n$-positive} if $\phi(M_n(\cA)^+)\subseteq M_n(\cB)^+$ 
and, it is called \emph{completely positive} if it is $n$-positive for all $n\in\NN$.

\begin{remarks}\label{r:cppsd}
Consider a linear map $\phi\colon\cA\ra\Lloc(\cH)$, for some locally 
$C^*$-algebra $\cA$ and some locally Hilbert space $\cH$. 

(1) The map $\phi$ is called \emph{$n$-positive semidefinite} if the kernel 
$\fk\colon\cA\times\cA\ra\Lloc(\cH)$ defined by
\begin{equation}\label{e:kabfi} 
\fk(a,b)=\phi(a^*b),\quad a,b\in\cA,
\end{equation}
is $n$-positive semidefinite in the sense of Subsection~\ref{ss:psk}, more precisely, 
for all $a_1,\ldots,a_n\in\cA$ and all $h_1,\ldots,h_n\in\cH$, we have
\begin{equation}\sum_{i,j=1}^n \langle \phi(a_i^*a_j)h_j,h_i\rangle_\cH\geq 0,
\end{equation}
and it is called \emph{positive semidefinite} if it is $n$-positive semidefinite for all $n\in\NN$.
Observing that
\begin{equation}
M_n(\Lloc(\cH))=M_n\otimes_*\Lloc(\cH)=\cB(\CC^n)\otimes_*\Lloc(\cH)
=\Lloc(\CC^n\otimes_{\mathrm{loc}}\cH),
\end{equation}
it follows that any positive semidefinite linear map $\phi\colon\cA\ra\Lloc(\cH)$ is completely
positive. Since any matrix $[a_{i,j}]_{i,j=1}^n\in M_n(\cA)^+$ is a linear combination of matrices
of type $[a_i^*a_j]_{i,j=1}^n$, it follows that the converse is true as well.

(2) Assume that $\cA=\varprojlim_{\lambda\in\Lambda}\cA_\lambda$ and 
$\cH=\varinjlim_{\lambda\in\Lambda}\cH_\lambda$, over the same directed poset $\Lambda$, and
that the linear map $\phi\colon\cA\ra\Lloc(\cH)$ is coherent in the sense of 
Subsection~\ref{ss:pllcs}, more precisely, there exists $\{\phi_\lambda\}_{\lambda\in\Lambda}$
with $\phi_\lambda\colon\cA\ra\cB(\cH_\lambda)$ linear map, for all $\lambda\in\Lambda$, such
that,
\begin{equation}
\pi_\lambda^{\Lloc(\cH)}\circ \phi=\phi_\lambda\circ\pi_\lambda^\cA,\quad \lambda\in\Lambda,
\end{equation}
where $\pi_\lambda^\cA\colon\cA\ra\cA_\lambda$ and 
$\pi_\lambda^{\Lloc(\cH)}\colon\Lloc(\cH)\ra\cB(\cH_\lambda)$ are the canonical $*$-morphisms.
In this case, $\phi$ is completely positive if and only if $\phi_\lambda$ is completely positive for all 
$\lambda\in\Lambda$. Since completely positive maps between $C^*$-algebras are 
automatically continuous, it follows that any coherent completely positive map 
$\phi\colon\cA\ra\Lloc(\cH)$ is continuous.

(3) If the completely positive map $\phi\colon\cA\ra\Lloc(\cH)$ is not coherent, it may happen
that it is not continuous. This is a consequence of the existence of $*$-morphisms between locally
$C^*$-algebras that are not continuous, cf.\ \cite{Phillips}.
\end{remarks}

Let $\phi\colon\cA\ra\Lloc(\cH)$
be a completely positive map, for some locally $C^*$-algebra $\cA$ and some locally Hilbert space
$\cH=\varinjlim_{\lambda\in\Lambda}\cH_\lambda$. By Remark~\ref{r:cppsd}, the kernel
$\fk\colon\cA\times\cA\ra\Lloc(\cH)$ defined as in
\eqref{e:kabfi} is positive semidefinite and observe that, when considering $\cA$ as a 
$*$-semigroup with respect to multiplication, it is invariant with respect to the left action of $\cA$
on itself, that is,
\begin{equation}\label{e:iph}
\fk(ab,c)=\phi((ab)^*c)=\phi(b^*a^*c)=\fk(b,a^*c),\quad a,b,c\in\cA.
\end{equation}
In order to apply Theorem~\ref{t:dilation}, the only obstruction is coming from condition (b).

We first make an additional assumption on $\phi$, namely that it is coherent, as in
Remark~\ref{r:cppsd}.(2). In particular, $\phi$ is continuous, cf.\ Remark~\ref{r:cppsd}.(3).
Depending on whether $\cA$ is unital or not, we distinguish two cases. If $\cA$ is unital, then fixing
$\lambda\in\Lambda$, one obtains the condition (b) due to the fact that $\cA_\lambda$ is a $C^*$-
algebra, e.g.\ see \cite{Arveson}. Briefly, for arbitrary
$a\in\cA_\lambda$, $b_1,\ldots,b_n\in\cA$ and $h_1,\ldots,h_n$ in $\cH$, since $\phi_\lambda$ is 
positive semidefinite, for any $y\in\cA_\lambda$ we have
\begin{equation}\label{e:sumay}\sum_{i,j=1}^n\langle\phi_\lambda(b_j^*y^*yb_i)h_i,h_j
\rangle_{\cH_\lambda}\geq 0.
\end{equation}
Without loss of generality we can assume that $\|a\|<1$ and let $y=(1-a^*a)^{1/2}\in\cA_\lambda$, 
hence from \eqref{e:sumay} it follows
\begin{equation}\label{e:inegafi}
\sum_{i,j=1}^n\langle\phi(b_i^*a^*ab_j)h_j,h_i\rangle_{\cH_\lambda}\leq \sum_{i,j=1}^n
\langle\phi(b_i^*b_j)h_j,h_i\rangle_{\cH_\lambda},
\end{equation} which proves that condition (b) holds, in this case.
Thus, we can apply Theorem~\ref{t:dilation} and get a locally Hilbert space 
linearisation $(\cK;V)$ of $\fk$, with 
$\cK=\varinjlim_{\lambda\in\Lambda}\cK_\lambda$ and $V\colon \cA\ra\Lloc(\cH,\cK)$ such that 
$V(b)^*V(c)=\fk(b,c)=\phi(b^*c)$ for all $b,c\in\cA$, as well as a $*$-representation 
$\pi\colon\cA\ra\Lloc(\cK)$ (this is indeed a $*$-representation of $*$-algebras 
since linearity comes 
for free), such that $\pi(a)V(b)=V(ab)$ for all $a,b\in\cA$. Since $\cA$ is unital, letting $W=V(1)
\in\Lloc(\cH,\cK)$, it follows that $\phi(a)=W^*\pi(a)W$, for all $a\in\cA$.

In case $\cA$ is not unital, one has to impose stronger assumptions. Firstly, the boundedness 
condition (b) can be proven:
with notation as in the proof of Theorem~\ref{t:dilation}, for a fixed $\lambda\in\Lambda$,
as in \eqref{e:pls}, one has a $*$-representation $\pi_\lambda \colon \cA\ra \cL(\cG_\lambda)$. 
Letting $\widetilde\cA=\cA\oplus\CC$ denote the unitisation of the $C^*$-algebra $\cA$, letting
$\widetilde\pi_\lambda\colon\widetilde\cA\ra\cL(\cG_\lambda)$ be defined by 
$\widetilde\pi_\lambda(a,t)=\pi_\lambda(a)+tI_{\cG_\lambda}$, $a\in\cA$, $t\in\CC$, we get a 
unital $*$-representation of $\widetilde\cA$ on the pre-Hilbert space $\cG_\lambda$,  in particular, 
$\widetilde\pi_\lambda$ maps unitary elements from $\widetilde\cA$ to unitary operators on 
$\cG_\lambda$. Since $\widetilde\cA$ is linearly generated by the set of its unitary elements, 
a standard argument, e.g.\ see \cite{Murphy}, proves the validity of the boundedness condition (b).

Secondly, recall that, according to 
a result in \cite{Inoue}, $\cA$ has approximate units. On $\Lloc(\cH,\cK)$
one introduces the \emph{strict topology}, also known as the \emph{so$^*$-topology}, which is the
locally convex topology defined by the family of seminorms 
$\Lloc(\cH)\ni T\mapsto \|T_\lambda h\|_{K_\lambda} +\|T_\lambda^*k\|_{\cH_\lambda}$, for all 
$\lambda\in\Lambda$, $h\in\cH_\lambda$, and $k\in\cK$,
where $T=\varprojlim T_\lambda$. It is easy to
see that $\Lloc(\cH,\cK)$ is complete with respect to the strict topology. Then, 
$\phi\colon \cA\ra\Lloc(\cH)$ is called \emph{strict} if, for some approximate unit $\{e_j\}_{j\in\cJ}$
of $\cA$, $\{\phi(e_j)\}_{j\in\cJ}$ is a Cauchy net with respect to the strict topology in $\Lloc(\cH)$.
Under the additional assumption that $\phi$ is strict, one proves, e.g.\ as in \cite{Murphy}, that
the net $\{V(e_j)\}_{j\in\cJ}$ is Cauchy with respect to the strict 
topology in $\Lloc(\cH,\cK)$, hence there exists $W\in\Lloc(\cH,\cK)$ such that 
$V(e_j)\xrightarrow[j\in \cJ]{ } W$, 
with respect to the strict topology. Again, we conclude that $\phi(a)=W^*\pi(a)W$ for all $a\in\cA$.

The preceding arguments prove a coherent 
version of the classical Stinespring Dilation Theorem \cite{Stinespring}.

\begin{theorem}\label{t:cstinespring}
Let $\phi\colon\cA\ra\Lloc(\cH)$ be a coherent linear map, for some 
locally $C^*$-algebra $\cA=\varprojlim_{\lambda\in\Lambda}\cA_\lambda$ 
and some locally Hilbert space 
$\cH=\varinjlim_{\lambda\in\Lambda}\cH_\lambda$. The following are equivalent:

\nr{1} $\phi$ is completely positive, and strict if $\cA$ is not unital.

\nr{2} There exists a locally Hilbert space $\cK=\varinjlim_{\lambda\in\Lambda}\cK_\lambda$,
a coherent $*$-representation $\pi\colon\cA\ra\Lloc(\cK)$, and
$W\in\Lloc(\cH,\cK)$, such that $\phi(a)=W^*\pi(a)W$ for all $a\in\cA$.
\end{theorem}

The second Stinespring type dilation theorem for locally bounded operator valued completely
positive maps on locally $C^*$-algebras, that we point out, says that in case $\phi$ is not coherent, 
one has to assume that it is continuous, and the same conclusion can be obtained (of course, less
the coherence of the $*$-representation $\pi$). This theorem is closer
to the Stinespring type theorems proven in \cite{Kasparov} and \cite{Joita2}, 
but rather different in nature.

\begin{theorem}\label{t:ncstinespring} Let $\phi\colon\cA\ra\Lloc(\cH)$ be a linear map, for some 
locally $C^*$-algebra $\cA$ and some locally Hilbert space 
$\cH=\varinjlim_{\lambda\in\Lambda}\cH_\lambda$. The following assertions are equivalent:

\nr{1} $\phi$ is a continuous completely positive map, and strict if $\cA$ is not unital.

\nr{2} There exists a locally Hilbert space $\cK=\varinjlim_{\lambda\in\Lambda}\cK_\lambda$,
a continuous $*$-representation $\pi\colon\cA\ra\Lloc(\cK)$, and
$W\in\Lloc(\cH,\cK)$, such that $\phi(a)=W^*\pi(a)W$ for all $a\in\cA$.
\end{theorem}

In order to prove Theorem~\ref{t:ncstinespring}, one has to take into account the continuity of 
$\phi$ in a slightly different fashion. Firstly, with notation as in Subsection~\ref{ss:lcsa},
in this case $\cA=\varprojlim_{p\in S(\cA)}\cA_p$, where $S(\cA)$, the collection of all continuous 
$C^*$-seminorms on $\cA$, is directed with respect to the order $p\leq q$ if $p(a)\leq q(a)$ for all 
$a\in\cA$. The main obstruction, when compared to the 
case of a coherent  completely positive map $\phi$ as before, comes from the fact that the two 
directed posets $\Lambda$ and $S(\cA)$ may be completely unrelated. In this case, one has to 
assume that the completely positive map $\phi\colon \cA\ra \Lloc(\cH)$ is continuous, hence, for 
any  $\lambda\in\Lambda$, there exists $p\in S(\cA)$ and $C_\lambda\geq 0$ such that
\begin{equation}
\|\phi(a)_\lambda\|_{\cH_\lambda}\leq C_\lambda\, p(a),\quad a\in\cA.
\end{equation}
A standard argument implies that $\phi$ factors to a completely positive map 
$\phi_\lambda \colon\cA_p\ra\cB(\cH_\lambda)$. To $\phi_\lambda$ one can apply a similar, but 
slightly more involved, procedure described before
for the case of a coherent completely positive map, to conclude that the boundedness condition (b) 
holds and, with a careful treatment of the two cases, either
$\cA$ is unital or $\cA$ is nonunital and $\phi$ a is
strict map, that there exists a continuous $*$-representation $\pi\colon\cA\ra\Lloc(\cK)$ and $W\in
\Lloc(\cH,\cK)$ such that $\phi(a)=W^*\pi(a)W$ for all $a\in \cA$. The technical details are 
very similar to, and to a certain extent simpler than, those in the proof of Theorem~3.5 in 
\cite{AyGheondea}, and we do not repeat them.

\section{Applications to Hilbert Locally $C^*$-Modules}\label{s:ahlcsm}

In this section, we show the main application of Theorem~\ref{t:dilation} to an operator model
with locally bounded operators for Hilbert modules over locally $C^*$-algebras and a direct construction of the exterior tensor product of Hilbert modules over locally $C^*$-algebras.

\subsection{Hilbert Locally $C^*$-Modules.} \label{ss:hlcsm}
We first briefly review the abstract concepts related to Hilbert modules over locally $C^*$-
algebras, see \cite{Mallios}, \cite{Phillips}, \cite{ZhuraevSharipov}. 
Let $\cA$ be a locally $C^*$-algebra and let $\cE$ be a
complex vector space. A paring $[\cdot,\cdot]\colon \cE\times\cE\ra \cA$ is called an 
\emph{$\cA$-valued gramian} or \emph{$\cA$-valued inner product} if
\begin{itemize}
\item[(g1)] $[e,e]\geq 0$ for all $e\in\cE$, and $[e,e]=0$ if and only if $e=0$.
\item[(g2)] $[e,\alpha g+\beta f]=\alpha [e,g]+\beta [e,f]$, for all $\alpha,\beta\in\CC$ and 
$e,f,g\in\cE$.
\item[(g3)] $[e,f]^*=[f,e]$ for all $e,f\in\cE$.
\end{itemize}
The vector space $\cE$ is called a \emph{pre-Hilbert locally $C^*$-module} if
\begin{itemize}
\item[(h1)] On $\cE$ there exists an $\cA$-gramian $[\cdot,\cdot]$, for some locally $C^*$-algebra 
$\cA$.
\item[(h2)] $\cE$ is a right $\cA$-module compatible with the $\CC$-vector space structure of $\cE$.
\item[(h3)] $[e,af]=[e,f]a$ for all $a\in\cA$ and all $e,f\in\cE$.
\end{itemize}
On any pre-Hilbert locally $C^*$-module $\cE$ over the locally $C^*$-algebra $\cA$, with 
$\cA$-gramian $[\cdot,\cdot]$, there exists a natural Hausdorff locally convex topology. More 
precisely, for any $p\in S(\cA)$, that is, $p$ is a continuous $C^*$-seminorm on $\cA$, letting
\begin{equation} \label{e:olp} 
\ol{p}(e)=p([e,e])^{1/2},\quad e\in\cE,
\end{equation} then $\ol{p}$ is a seminorm on $\cE$. If the  topology generated on $\cE$ by 
$\{\ol{p}\mid p\in S(\cA)\}$ is complete, then $\cE$ is called a \emph{Hilbert locally $C^*$-module}. 
In case $\cA$ is a $C^*$-algebra, we talk about a \emph{Hilbert $C^*$-module} $\cE$, with norm
$\cE\ni e\mapsto \|[e,e]\|_\cA^{1/2}$.

Let $\cE$ be a Hilbert module over a locally $C^*$-algebra $\cA$ and, for $p\in S(\cA)$, recall
that $\cI_p$, defined as in \eqref{e:cisp}, is a closed $*$-ideal of $\cA$ with respect to which 
$\cA_p=\cA/\cI_p$ becomes a $C^*$-algebra under the canonical $C^*$-norm $\|\cdot\|_p$ defined 
as in \eqref{e:tlp}. Considering
\begin{equation}%
\cN_p=\{e\in\cE \mid [e,e]\in\cI_p\},
\end{equation} then $\cN_p$ is a closed $\cA$-submodule of $\cE$ and
$\cE_p=\cE/\cN_p$ is a Hilbert module over the $C^*$-algebra $\cA_p$, with norm
\begin{equation}\label{e:nhp} 
\|e+\cN_p\|_{\cE_p}=\inf_{f\in \cN_p} \ol{p}(e+f)=\ol{p}(e),\quad e\in\cE.
\end{equation}
For each $p,q\in S(\cA)$ with $p\leq q$, observe that $\cN_q\subseteq\cN_p$ and hence, there 
exists a canonical projection $\pi_{p,q}\colon \cE_q\ra\cE_p$, $\pi_{p,q}(e+\cN_q)=e+\cN_p$, 
$h\in\cE$, and $\pi_{p,q}$ is an $\cA$-module map, such that 
$\|\pi_{p,q}(e+\cN_q)\|_{\cE_p}\leq \|e+\cN_p\|_{\cE_q}$ for all $e\in\cE$. In addition, 
$\{\cE_p\}_{p\in S(\cA)}$ and $\{\pi_{p,q}\mid p,q\in
S(\cA), \ p\leq q\}$ make a projective system of Hilbert $C^*$-modules and
$\cE=\varprojlim\limits_{p\in S(\cA)} \cE_p$. 

\begin{examples}\label{ex:rhlm} 
(1) Let $\cH=\varinjlim\limits_{\lambda\in\Lambda} \cH_\lambda$ and
$\cK=\varinjlim\limits_{\lambda\in\Lambda} \cK_\lambda$ be two locally Hilbert spaces
with respect to the same directed poset $\Lambda$. 
We consider $\Lloc(\cH)$ as a locally $C^*$-algebra as in Example~\ref{ex:lcalhs} (1). Observe that 
the vector space $\Lloc(\cH,\cK)$, see Subsection~\ref{ss:lbo}, has a natural structure of 
right $\Lloc(\cH)$-module which is compatible with the $\CC$-vector space structure 
of $\Lloc(\cH,\cK)$ and, considering the gramian $[\cdot,\cdot]_{\Lloc(\cH,\cK)}$ defined by
\begin{equation}\label{e:glhk} 
[T,S]_{\Lloc(\cH,\cK)}=T^*S,\quad T,S\in\Lloc(\cH,\cK),
\end{equation} $\Lloc(\cH,\cK)$ becomes a pre-Hilbert module over the locally $C^*$-algebra 
$\Lloc(\cH)$. 

The complex vector space $\Lloc(\cH,\cK)$ has a natural family of seminorms
\begin{equation}\label{e:slhk} 
q_{\mu}(T)=\|T_\mu\|_{\cB(\cH_\mu,\cK_\mu)},\quad 
T=\varprojlim_{\lambda\in\Lambda} T_\lambda\in\Lloc(\cH,\cK),\ \mu\in\Lambda.
\end{equation}
Observe that, with respect to the $C^*$-seminorms $p_\mu$ on $\Lloc(\cH)$, defined 
at \eqref{e:psmu}, we have
\begin{equation*} q_\mu(T)^2=\|T_\mu\|^2_{\cB(\cH_\mu,\cK_\mu)}
=\|T^*_\mu T_\mu\|_{\cB(\cH_\mu)}
=p_\mu([T,T]_{\Lloc(\cH,\cK)}),\ \mu\in\Lambda,\ 
T=\varprojlim_{\lambda\in\Lambda} T_\lambda\in\Lloc(\cH,\cK),
\end{equation*}
hence, compare with \eqref{e:olp},
the collection of seminorms $\{q_\mu\}_{\mu\in\Lambda}$ defines exactly the canonical
topology on the pre-Hilbert locally $C^*$-module $\Lloc(\cH,\cK)$. 
Since, as easily observed, this locally
convex topology is complete on $\Lloc(\cH,\cK)$, it follows that $\Lloc(\cH,\cK)$ is a Hilbert locally
$C^*$-module over $\Lloc(\cH)$.

(2) With notation as in item (1), let $\cA$ be a closed $*$-subalgebra of $\Lloc(\cH)$, considered
as a locally $C^*$-algebra as in Example~\ref{ex:lcalhs} (2). Let $\cE$ be a closed vector subspace
of $\Lloc(\cH,\cK)$ that is an $\cA$-module and such that $T^*S\in\cA$ for all $T,S\in\cE$. Then, 
the definition in \eqref{e:glhk} provides a gramian $[T,S]_\cE=T^*S$, $T,S\in\cE$, which 
turns $\cE$ into a Hilbert locally $C^*$-module over $\cA$. Observe that the embedding
of $\cE$ in $\Lloc(\cH,\cK)$ is a coherent linear map.
\end{examples}

A Hilbert locally $C^*$-module $\cE$ as in Example~\ref{ex:rhlm} (2) is called a \emph{represented
Hilbert locally $C^*$-module} or, a \emph{concrete Hilbert locally $C^*$-module}.

\begin{theorem}\label{t:rephlcm}
Let $\cE$ be a Hilbert module over some locally $C^*$-algebra $\cA$. Then, $\cE$ is
isomorphic to a concrete Hilbert locally $C^*$-module, more precisely, there exist two
locally Hilbert spaces $\cH=\varinjlim\limits_{p\in S(\cA)} \cH_p$ and 
$\cK=\varinjlim\limits_{p\in S(\cA)} \cK_p$,
a coherent faithful $*$-morphism $\phi\colon\cA\ra\Lloc(\cH)$, and 
a coherent one-to-one linear map $\Phi\colon\cE\ra\Lloc(\cH,\cK)$ such that:
\begin{itemize} 
\item[(i)] $\Phi(e)^*\Phi(f)=\phi([e,f]_\cE)$ for all $e,f\in \cE$. 
\item[(ii)] $\Phi(ea)=\Phi(e)\phi(a)$ for all $e\in\cE$ and all $a\in\cA$.
\end{itemize}
\end{theorem}

\begin{proof} We consider the canonical representation of the locally 
$C^*$-algebra $\cA=\varprojlim\limits_{p\in S(\cA)} \cA_p$, as in \eqref{e:alim}.
By Theorem~\ref{t:gelfandnaimark}, there exists a locally Hilbert space 
$\cK=\varinjlim\limits_{p\in S(\cA)}\cH_p$, for some strictly injective system of Hilbert spaces
$\{\cH_p\}_{p\in S(\cA)}$, and a coherent $*$-monomorphism $\phi\colon\cA\ra\Lloc(\cH)$, 
more precisely, for each $p\in S(\cA)$ there exists a faithful $*$-morphism 
$\phi_p\colon\cA_p\ra\Lloc(\cH_p)$, such that $\phi=\varprojlim\limits_{p\in S(\cA)}\phi_p$.
Consider the $\Lloc(\cH)$-valued kernel $\fk\colon \cE\times \cE\ra\Lloc(\cH)$ defined by
\begin{equation}\fk(e,f)=\phi([e,f]_\cE),\quad e,f\in\cE.
\end{equation} We claim that $\fk$ is locally positive semidefinite. To see this, let $n\in\NN$, 
$e_1,\ldots,e_n\in\cE$, and $a_1,\ldots,a_n\in\cA$. Then
\begin{align*}\sum_{i,j=1}^n \phi(a_i)^*\fk(e_i,e_j)\phi(a_j)
& =\sum_{i,j=1}^n \phi(a_i)^* \phi([e_i,e_j]_\cE)\phi(a_j) = \sum_{i,j}^n \phi(a_i^*[e_i,e_j]_\cE\, a_j) \\
& = \phi([\sum_i^n a_ie_i,\sum_{j=1}^ne_ja_j]_\cE)\geq 0.
\end{align*}
This implies that, for any $p\in S(\cA)$, the kernel $\fk_p$ has the following property: for any
$n\in\NN$, $e_1,\ldots,e_n\in\cE$, and $h_1,\ldots,h_n$ in the closed linear span of 
$\phi_p(\cA_p)\cH_p$ in $\cH_p$, we have
\begin{equation}\label{e:pp}
\sum_{i,j=1}^ n\langle\fk_p(e_i,e_j)h_j,h_i\rangle_{\cH_p}\geq 0.
\end{equation} 
Since, for arbitrary $e,f\in\cE$, the closed linear span of  $\phi_p(\cA_p)\cH_p$ is reducing 
$\fk_p(e,f)$ and $\fk_p(e,f)h=0$ for all $h\in\cH_p$
and orthogonal onto $\phi_p(\cA_p)\cH_p$, it follows that, actually, the inequality 
\eqref{e:pp} is true for all $h_1,\ldots,h_n\in\cH_p$. Consequently, $\fk_p$ is a positive semidefinite 
kernel for all $p\in S(\cA)$, hence $\fk$ is a locally positive semidefinite kernel.

We can now apply Theorem~\ref{t:dilation}, for a trivial $*$-semigroup $ S=\{\epsilon\}$,
and get a locally Hilbert space linearisation $(\cK;\Phi)$ of the kernel $\fk$, with a locally Hilbert 
space 
$\cK=\varinjlim\limits_{p\in S(\cA)}\cK_p$ and $\Phi=\varprojlim\limits_{p\in S(\cA)} \Phi_p$, 
where, for each $p\in S(\cA)$, $\Phi_p\colon \cE\ra\cB(\cH_p,\cK_p)$ has the property 
\begin{equation*}\Phi_p(e)^*\Phi_p(f)=\fk_p(e,f)=\phi_p([e,f]_\cE),\quad e,f\in\cE. 
\end{equation*}
This proves (i).

Inspecting the proof of Theorem~\ref{t:dilation}, in particular \eqref{e:vex}, it follows that the map 
$\Phi\colon \cE\ra\Lloc(\cH,\cK)$ is defined by
\begin{equation*}
(\Phi(e)h)(f)=\fk(f,e)h=\phi([f,e]_\cE)h,\quad e,f\in\cE,\ h\in\cH,
\end{equation*} and hence is linear. Moreover, for any $e,f\in\cE$, $a\in\cA$, and $h\in\cH$ 
we have
\begin{equation*}(\Phi(ea)h)(f)=\phi([f,e\,a]_\cE)h=\phi([f,e]_\cE\, a)h=\phi([f,e]_\cE)\phi(a)h
=(\Phi(e)\phi(a)h)(f),
\end{equation*} and hence (ii) is proven.
\end{proof}

\subsection{The Exterior Tensor Product of Hilbert Locally $C^*$-Modules.}
Let $\cA$ and $\cB$ be two locally $C^*$-algebras and let $\cE$ and $\cF$ be
two Hilbert locally $C^*$-modules 
over $\cA$ and, respectively, $\cB$. Let $\cA\otimes_*\cB$ denote the spatial $C^*$-algebra tensor
product, see Subsection~\ref{ss:stplcsa}. 
Consider the algebraic tensor product $\cE\otimes_{\mathrm{alg}}\cF$ of the vector 
spaces $\cE$ and $\cF$ and observe that there is a natural right action of the (algebraic)
tensor product $*$-algebra $\cA\otimes_{\mathrm{alg}}\cB$, first defined on elementary tensors
\begin{equation}\label{e:tap}
(e\otimes f)(a\otimes b)=(ea)\otimes(fb),\quad a\in\cA,\ b\in\cB,\ e\in\cE,\ f\in\cF,
\end{equation} and then extended by linearity, hence $\cE\otimes_{\mathrm{alg}}\cF$ is naturally 
an $\cA\otimes_{\mathrm{alg}}\cB$-module. Also, there is an 
$\cA\otimes_{\mathrm{alg}}\cB$-valued
pairing on $\cE\otimes_{\mathrm{alg}}\cF$, first defined on elementary tensors
\begin{equation}\label{e:tpp}
[e_1\otimes f_1,e_2\otimes f_2]=[e_1,f_1]\otimes [e_2,f_2],\quad e_1,e_2\in\cE,\ f_1,f_2\in\cF,
\end{equation} and then extended by linearity.

\begin{theorem}\label{t:etp} 
With notation as before, the pairing defined at \eqref{e:tpp} is uniquely extended to
an $\cA\otimes_*\cB$-gramian on $\cE\otimes_{\mathrm{alg}}\cF$, with respect to which it
is a pre-Hilbert locally $C^*$-module,
and then it is uniquely extended
to the completion of $\cE\otimes_{\mathrm{alg}}\cF$ to a Hilbert module over the
locally $C^*$-algebra $\cA\otimes_*\cB$.
\end{theorem}

\begin{proof} By Theorem~\ref{t:gelfandnaimark},
as in Subsection~\ref{ss:stplcsa}, without loss of generality we can assume that, for 
two locally Hilbert spaces $\cH=\varinjlim_{\lambda\in\Lambda}\cH_\alpha$ and 
$\cG=\varinjlim_{\alpha\in A}\cG_\alpha$, $\cA$ is a locally $C^*$-subalgebra 
of $\Lloc(\cH)$ and $\cB$ is
a locally $C^*$-subalgebra of $\Lloc(\cG)$. These yield a natural embedding of the spatial
tensor product of locally $C^*$-algebras $\cA\otimes_*\cB$ into 
$\Lloc(\cH\otimes_{\mathrm{loc}}\cG)$.

Then, by Theorem~\ref{t:rephlcm}, without loss of 
generality we can assume that $\cE$ is an $\cA$-submodule of $\Lloc(\cH,\cK)$, for 
$\cK=\varinjlim_{\lambda\in\Lambda}\cK_\lambda$ some locally Hilbert space, and $\cF$ is
an $\cB$-submodule of $\Lloc(\cG,\cN)$, for $\cN=\varinjlim_{\alpha\in A} \cN_\alpha$ some
locally Hilbert space.

We consider the locally Hilbert tensor products $\cH\otimes_{\mathrm{loc}}\cG$ and
$\cK\otimes_{\mathrm{loc}}\cN$, as in Subsection~\ref{ss:tplhs}, and observe that 
$\cE\otimes_{\mathrm{alg}}\cF$ is naturally included in $\Lloc(\cH\otimes_{\mathrm{loc}}\cG,
\cK\otimes_{\mathrm{loc}}\cN)$.
Then observe that $\{\cB(\cH_\lambda,\cK_\lambda)\otimes\cB(\cG_\alpha,\cN_\alpha)
\}_{(\lambda,\alpha)\in\Lambda\times A}$ is a projective system of Banach spaces whose 
projective limit
\begin{equation}
\Lloc(\cH,\cK)\otimes_{\mathrm{loc}}\Lloc(\cG,\cN)
=\varprojlim_{(\lambda,\alpha)\in\Lambda\times A} \cB(\cH_\lambda,\cK_\lambda)
\otimes\cB(\cG_\alpha,\cN_\alpha),
\end{equation} is naturally organised as a Hilbert module over the locally $C^*$-algebra 
$\Lloc(\cH)\otimes_{\mathrm{loc}}\Lloc(\cG)$. Consequently, we perform the extension of the pairing defined at 
\eqref{e:tpp} to
an $\cA\otimes_*\cB$-gramian on $\cE\otimes_{\mathrm{alg}}\cF$, with respect to which it
is a pre-Hilbert locally $C^*$-module,
and then we uniquely extend it
to the closure of $\cE\otimes_{\mathrm{alg}}\cF$ in $\Lloc(\cH,\cK)\otimes_{\mathrm{loc}}
\Lloc(\cG,\cN)$, as a Hilbert module over the
locally $C^*$-algebra $\cA\otimes_*\cB$.
\end{proof}

The tensor product $\cE\otimes_{\mathrm{ext}} \cF$, defined as the completion of $\cE\otimes_{\mathrm{alg}}\cE$ 
and organised by Theorem~\ref{t:etp} as a Hilbert module over the locally $C^*$-algebra 
$\cA\otimes_*\cB$,
is called the \emph{exterior tensor product} of the Hilbert locally $C^*$-modules $\cE$ and $\cF$,
and it coincides with that obtained in \cite{Joita4}, see also \cite{Lance}.

\end{document}